\documentclass{amsart}
\usepackage{amssymb}
\usepackage{mathrsfs}
\usepackage{amsmath}
\usepackage{amsfonts}
\usepackage{pifont}
\usepackage{graphicx,amssymb,mathrsfs,amsmath}
\usepackage{tocvsec2}
\usepackage{color}
\newtheorem{thm}{Theorem}[section]
\newtheorem{cor}[thm]{Corollary}
\newtheorem{lem}[thm]{Lemma}
\newtheorem{prop}[thm]{Proposition}
\theoremstyle{definition}
\newtheorem{defn}[thm]{Definition}
\newtheorem{example}[thm]{Example}
\theoremstyle{remark}
\newtheorem{rem}[thm]{Remark}
\numberwithin{equation}{section}

\begin{document}
\title[Disjoint distributional chaos in Fr\' echet spaces]{Disjoint distributional chaos in Fr\' echet spaces}

\author{Marko Kosti\' c}
\address{Faculty of Technical Sciences,
University of Novi Sad,
Trg D. Obradovi\' ca 6, 21125 Novi Sad, Serbia}
\email{marco.s@verat.net}

{\renewcommand{\thefootnote}{} \footnote{
\\ \text{  }  \ \    {\it Key words and phrases.} Disjoint distributional chaos; disjoint irregular vectors; multivalued linear operators; backward shift operators; weighted translation operators on locally compact groups; Fr\' echet spaces.
\\  \text{  }  \ \ The author was partially supported by grant 174024 of Ministry of Science and Technological Development, Republic of Serbia.}}

\begin{abstract}
We introduce several different notions of disjoint distributional chaos for sequences of multivalued linear operators in Fr\'echet spaces. Any of these notions seems to be new
and not considered elsewhere even for linear continuous operators in Banach spaces. We focus special attention to the analysis of some specific classes
of linear continuous operators having a certain disjoint distributionally chaotic behaviour, providing also a great number of  illustrative examples and applications of our abstract theoretical results.
\end{abstract}
\maketitle

\section{Introduction and Preliminaries}\label{intro}

Let $X$ be a separable Fr\' echet space. A
linear operator $T$ on $X$ is said to be hypercyclic iff there
exists an element $x\in D_{\infty}(T)\equiv \bigcap_{n\in {\mathbb
N}}D(T^{n})$ whose orbit $\{ T^{n}x : n\in {{\mathbb N}}_{0} \}$ is
dense in $X;$ $T $ is said to be topologically transitive,
resp. topologically mixing, iff for every pair of open
non-empty subsets $U,\ V$ of $X,$ there exists $n_{0}\in {\mathbb
N}$ such that $T^{n_{0}}(U) \ \cap \ V \neq \emptyset ,$ resp. there
exists $n_{0}\in {\mathbb N}$ such that, for every $n\in {\mathbb
N}$ with $n\geq n_{0},$ $T^{n}(U) \ \cap \ V \neq \emptyset .$ Finally, $T$ is said to be chaotic iff it is topologically transitive and the set of periodic points of $T,$ defined by $P_{T}:=\{x\in D_{\infty}(T) : (\exists n\in {\mathbb N})\, T^{n}x=x\},$ is dense in $X.$

The notion of distributional chaos for interval maps was introduced for the first time by B. Schweizer and J. Sm\' ital in \cite{smital} (this type of chaos was called strong chaos there, 1994). In the setting of linear continuous operators, 
distributional chaos was firstly considered in the research studies of quantum harmonic oscillator, by J. Duan et al \cite{duan} (1999) and P. Oprocha \cite{p-oprocha} (2006). The first systematic study of distributional chaos for linear continuous operators in Fr\' echet spaces was conducted by 
N. C. Bernardes Jr. et al \cite{2013JFA} (2013), while the first systematic study of distributional chaos for linear, not necessarily continuous, operators in Fr\' echet spaces was conducted by J. A. Conejero et al \cite{mendoza} (2016). Further information about distributional chaos in metric and Fr\' echet spaces can be obtained by consulting \cite{2011}, \cite{2018JMAA}-\cite{ARXIV}, \cite{chen-chen}, \cite{gimenez-p}-\cite{gimenez}, \cite{fu-tan}-\cite{fu-wang}, \cite{marek-trio}, \cite{luo}, \cite{p-oprocha-tams}-\cite{p-oprocha}, \cite{smital-duo}, \cite{tan-fu} and references cited therein.

On the other hand, the notion of disjoint hypercyclicity in linear topological dynamics was introduced independently by L. Bernal--Gonz\'alez \cite{bg07} (2007) and J. B\`es, A. Peris \cite{bp07} (2007). From then on, a great number of authors has analyzed this important notion and similar concepts. For more details about disjoint hypercyclic operators and their generalizations, one may refer e.g. to \cite{bms14}, \cite{bm201345}-\cite{chen-marek}, \cite{jusef}, \cite{knjigaho}-\cite{FKP}, \cite{liang-dis-super}, \cite{ma10}-\cite{om-rs}, \cite{y-puig}-\cite{salasinjo}, \cite{shkarin} and references cited therein.

The basic facts about topological dynamics of linear continuous operators in Banach and Fr\' echet spaces can be obtained by consulting the monographs \cite{bayart} by F. Bayart, E. Matheron and
\cite{erdper} by K.-G. Grosse-Erdmann, A. Peris. In a joint research study with J. A. Conejero, C.
-C. Chen and M. Murillo-Arcila \cite{kerry-drew}, the author has recently introduced and analyzed a great deal of topologically dynamical properties for multivalued linear operators (see also \cite{abakumov}, \cite{kd-prim} and \cite{bohemica}). Concerning distributional chaos and its generalizations for sequences of multivalued linear operators in Fr\' echet spaces, the first step has been made recently by M. Kosti\' c in \cite{marek-trio}.

Up to now, it was not clear how to define the notion of disjoint distributional chaos (in both, linear and non-linear setting). The main aim of this paper is to propose several different definitions of (subspace) disjoint distributional chaos for multivalued linear operators and their sequences in both, finite and infinite dimensional Fr\' echet spaces (the study of finite dimensional spaces is important and gives new theoretical insights at the notion of distributional chaos that has not been considered in \cite{2013JFA} and \cite{mendoza}; see Example \ref{totan} and Example \ref{totanr} below). As mentioned in the abstract, any of these notions of disjoint distributional chaos seems to be new even for linear continuous operators acting on Banach function spaces. 
We apply our abstract results to a great number of concrete classes of operators, like the backward shift operators on Fr\' echet sequence spaces, weighted translation operators on Orlicz spaces of locally compact groups and bounded differential operators on Fr\' echet spaces of entire functions. 
We also present numerous examples of disjoint distributionally chaotic linear unbounded differential operators, continuing thus our research study raised in \cite{mendoza}, and disjoint distributionally chaotic multivalued linear operators. For the sake of brevity and better exposition, disjoint distributionally chaotic extensions of multivalued linear
operators (cf. \cite{kerry-drew}-\cite{kd-prim} and \cite{marek-faac} for similar concepts), disjoint distributionally chaotic properties of continuous operators on metric spaces (cf. \cite{fu-tan}-\cite{fu-wang}, \cite{p-oprocha-tams}, \cite{smital-duo} and \cite{tan-fu} for related references concerning distributional chaos) as well as  
distributionally chaotic properties of abstract (degenerate) partial differential equations with integer or fractional order time-derivatives will not be considered within the framework of this paper (cf. \cite{angel}, \cite{bara}-\cite{bara-duo}, \cite{mendoza}, \cite{FKP}-\cite{marek-nsjom} and references quoted therein for further information in this direction).

The organization and main ideas of this paper can be described as follows. After collecting some preliminary results and facts about Fr\' echet spaces and upper densities, in Section \ref{MLOs} we recall the basic facts and definitions from the theory of multivalued linear operators (MLOs, in the sequel) that will be necessary for our further work; in a separate subsection, we consider distributionally chaotic properties of multivalued linear operators and their sequences. The first aim of third section is to fix the notion of $(d,\tilde{X},i)$-distributional chaos for 
MLOs and their sequences; here, $\tilde{X}$ is a closed linear subspace of $X$ and
$i\in {\mathbb N}_{12}$  (see Definition \ref{DC-unbounded-fric-DISJOINT}). In other words, we introduce and analyze twelve different types of disjoint distributional chaos here.
The notion of $(d,\tilde{X},1)$-distributional chaos
is the strongest one, while the notion of
$(d,\tilde{X},i)$-distributional chaos for $i\in \{2,3,7\}$ is incredibly important because this type of disjoint distributional chaos implies the $\tilde{X}$-distributional chaos of any single component (that is a sequence of MLOs, in general) under our consideration;
$(d,\tilde{X},i)$-distributional chaos for $i\in \{4,5,6,8\}$  implies the existence of a single component that is $\tilde{X}$-distributional chaos (see Proposition \ref{trick} and Remark \ref{prc-qwe}(i)). Besides that, we have found the notion of $(d,\tilde{X},9)$-distributional chaos very intriguing because the existence of this type of disjoint distributional chaos has some obvious connections with the existence of $\tilde{X}$-distributional chaos for the induced diagonal mapping $X\rightarrow Y^{N}$ (see Proposition \ref{prc-qwe} for more details).
Before we go any further, we need to say that other types of $(d,\tilde{X},i)$-distributional chaos, although introduced, are very mild and 
interested only from some theoretical point of view; it is also worth saying that we will not discuss here the notion of full $(d,\tilde{X},i)$-distributional chaos, in which the corresponding $\sigma$-scramled set $S$ can be chosen to be the whole manifold $\tilde{X}$ (cf. \cite{bara} and \cite{gimenez} for some results established in this direction).

In our approach, we work with four different types of disjoint distributionally unbounded vectors and 
four different types of disjoint distributionally near to $0$ vectors (see Definition \ref{DC-zero-DISJOINT} and Definition \ref{DC-unbounded-DISJOINT}), while for each type of 
$(d,\tilde{X},i)$-distributional chaos there exists exactly one type of disjoint distributionally irregular vectors accompanied (see Definition \ref{iregularni-prc}).
Concerning the existence of distributionally unbounded vectors, in multivalued linear setting we recognise some important differences between Banach spaces and Fr\' echet spaces (see Example \ref{gos}).
After consideration in this example, we first note that some very weak forms of disjoint distributional chaos can be analyzed for MLOs and their sequences by slightly modifying the notion from Definition \ref{DC-unbounded-fric-DISJOINT} (see Example \ref{deds-MLOS-vg}
for an illustration of this fact for purely MLOs). Following the approach obeyed in the paper \cite{bp07} by J. B\`es and A. Peris, we introduce and explain the importance of (uniformly) $(d,\tilde{X},i)$-distributionally irregular manifolds for MLOs.
In Example \ref{2}-Example \ref{12}, we present several counterexamples showing that
the notions of $(d,i_{1})$-distributional chaos and $(d,i_{2})$-distributional chaos are not the same for different values of indexes $i_{1}$ and $i_{2}.$ We close the third section of paper by stating Proposition \ref{subspace-fric-DISJOINT}, which rewords a corresponding result from \cite{mendoza} for disjointness and which particularly shows that the case $\tilde{X}=X$ can be always assumed in a certain sense. 

In the fourth section of paper, which contains our main theoretical contributions, we continue our analyses from \cite{mendoza} by enquiring into certain possibilities to extend the structural results from the foundational paper  
\cite{2013JFA} by N. C. Bernardes Jr. et al. Our first structural result is Theorem \ref{sequences}, in which we prove
the existence of a
$(d,1)$-distributionally irregular vector for a corresponding sequence
$((T_{j,k})_{k\in {\mathbb N}})_{1\leq j\leq N}$ of linear continuous operators acting between not necessarily the same pivot Fr\' echet spaces (in this section, we consider only single-valued linear operators but the continuity is neglected sometimes). The proof of this result leans heavily on the use
of the arguments contained in the proofs of \cite[Proposition 7, Proposition 9]{2013JFA}.
It is worth noting that Theorem \ref{sequences} can be formulated for any other type of $(d,i)$-distributionally irregular vectors.
We later observe that
the proof of implication (iv) $\Rightarrow$ (iii) in \cite[Theorem 15]{2013JFA} (cf. also the second part of \cite[Theorem 3.7]{kerry-drew}) can be 
used for proving the dense $(d,1)$-distributional chaos of linear continuous operators (see Theorem \ref{mmnn}). For linear operators, 
we can use Theorem \ref{sequences} and Theorem \ref{mmnn} as well as the well-known process of regularization to deduce new important results, Corollary \ref{cea1} and Corollary \ref{mmn} (it seems that the theory of $C$-regularized semigroups has to be applied here; see the paper \cite{ralphchaos} by R. deLaubenfels, H. Emamirad and K.-G. Grosse--Erdmann for an initial idea in this direction, and
\cite{knjigah}-\cite{FKP} for further information concerning $C$-regularized semigroups). Some applications of the above results to differential operators are presented in Example \ref{qaz} and Example \ref{lupa}. Motivated by the research of V. M\"uller and J. Vr\v sovsk\'y \cite{milervrs}, at the end of this section
we analyze the existence of various types of $(d,i)$-distributionally unbounded vectors for the sequences of linear operators (see Proposition \ref{orbits-rija} and Proposition \ref{owq-prcko} for bounded operators in Banach spaces, as well as Corollary \ref{ora-rija} and Corollary \ref{ora-rija-3} for unbounded operators in Banach spaces). We present an instructive application to unbounded backward shift operators in Theorem \ref{rikardinjo} and Example \ref{primena-shifts} (to our best knowledge, (disjoint) distributional chaos for such operators has not been examined elsewhere by now).

The fifth section of paper aims to investigate some special classes of operators having a certain disjoint distributionally chaotic
behaviour. In particular, we analyze backward shift operators on Fr\' echet sequence spaces and weighted translation operators on Orlicz spaces of locally compact groups (in this section, we deduce several proper extensions of results from \cite[Section 4]{2013JFA}). 
It is shown that the $(d,1)$-distributional chaos of a tuple of unilateral backward shift operators is not equivalent with distributional chaos of single components taken separately (see Example \ref{bruk}).


We use the standard notation throughout the paper. We assume that $X$ and $Y$ are two
non-trivial Fr\' echet space over the same field of scalars ${\mathbb K}\in \{ {\mathbb R}, {\mathbb C} \}$ as well as that the topologies of $X$ and $Y$ are 
induced by the fundamental systems $(p_{n})_{n\in {\mathbb N}}$ and $(p_{n}^{Y})_{n\in {\mathbb N}}$ of
increasing seminorms, respectively (separability of $X$ or $Y$ is not assumed a priori in future). Then the translation invariant metric\index{ translation invariant metric} $d :
X\times X \rightarrow [0,\infty),$ defined by
\begin{equation}\label{metri}
d(x,y):=\sum
\limits_{n=1}^{\infty}\frac{1}{2^{n}}\frac{p_{n}(x-y)}{1+p_{n}(x-y)},\
x,\ y\in X,
\end{equation}
enjoys the following properties:
$
d(x+u,y+v)\leq d(x,y)+d(u,v),$ $x,\ y,\ u,\
v\in X;
$
$d(cx,cy)\leq (|c|+1)d(x,y),$ $ c\in {\mathbb K},\ x,\ y\in X,
$
and
$
d(\alpha x,\beta x)\geq \frac{|\alpha-\beta|}{1+|\alpha-\beta|}d(0,x),$ $x\in X,$ $ \alpha,\ \beta \in
{\mathbb K}.$
For any $\epsilon>0$ given in advance, set $L(0,\epsilon):=\{x\in X : d(x,0)<\epsilon\}.$
Define the translation invariant metric $d_{Y} :
Y\times Y \rightarrow [0,\infty)$ by replacing $p_{n}(\cdot)$ with
$p_{n}^{Y}(\cdot)$ in (\ref{metri}).
If
$(X,\|\cdot \|)$ or $(Y,\|\cdot \|_{Y})$ is a Banach space, then we assume that the
distance of two elements $x,\ y\in X$ ($x,\ y\in Y$) is given by $d(x,y):=\|x-y\|$ ($d_{Y}(x,y):=\|x-y\|_{Y}$).
Keeping in mind this terminological change,
our structural results clarified in Fr\' echet spaces continue to hold in the case that $X$ or $Y$ is a Banach space. 

Throughout this paper, it will be assumed that $N\in {\mathbb N}$ and $N\geq 2.$ Then the fundamental system of increasing seminorms $({\bf p}_{n}^{Y^{N}})_{n\in {\mathbb N}},$ 
where ${\bf p}_{n}^{Y^{N}}(x_{1},\cdot \cdot \cdot,x_{N}):=\sum_{j=1}^{N}p_{n}^{Y}(x_{j}),$ $n\in {\mathbb N}$ ($x_{j}\in Y$ for $1\leq j\leq N$), induces the topology on the Fr\' echet space $Y^{N}.$ The translation invariant metric
\begin{align*}
{\rm d}_{Y^{N}}(\vec{x},\vec{y}):=\sum
\limits_{n=1}^{\infty}\frac{1}{2^{n}}\frac{{\bf p}_{n}(\vec{x}-\vec{y})}{1+{\bf p}_{n}(\vec{x}-\vec{y})},\quad
\vec{x},\ \vec{y}\in Y^{N},
\end{align*}
is strongly equivalent with the metric 
$$
d_{Y^{N}}(\vec{x},\vec{y}):=\max_{1\leq j\leq N}d_{Y}(x_{j},y_{j}),\quad \vec{x}=(x_{1},\cdot \cdot \cdot,x_{N})\in Y^{N},\ \vec{y}=(y_{1},\cdot \cdot \cdot,y_{N})\in Y^{N},
$$
since
\begin{align}\label{metri-sa}
d_{Y^{N}}(\vec{x},\vec{y})\leq {\rm d}_{Y^{N}}(\vec{x},\vec{y}) \leq N^{2}d_{Y^{N}}(\vec{x},\vec{y}),\quad \vec{x}\in Y^{N},\ \vec{y}\in Y^{N}.
\end{align}
For the sake of completeness, we will prove the second inequality in \eqref{metri-sa}. Let the vectors $\vec{x}\in Y^{N},$ $\vec{y}\in Y^{N}$ be given, and let for each $n\in {\mathbb N}$ the number
$j_{n}\in {\mathbb N}_{N}$ be such that $p_{n}(x_{j_{n}}-y_{j_{n}})=\max_{1\leq j\leq N}p_{n}(x_{j}-y_{j}).$
Then 
\begin{align*}
{\rm d}_{Y^{N}}(\vec{x},\vec{y})\leq
\sum \limits_{n=1}^{\infty}\frac{N}{2^{n}}\frac{p_{n}(x_{j_{n}}-y_{j_{n}})}{1+p_{n}(x_{j_{n}}-y_{j_{n}})}.
\end{align*}
Set $M_{j}:=\{n\in {\mathbb N} : j_{n}=j\}$ ($j\in {\mathbb N}_{N}$).
Then $\bigcup_{j\in {\mathbb N}_{N}}M_{j}={\mathbb N}$ and the above inequality yields
\begin{align*}
{\rm d}_{Y^{N}}&(\vec{x},\vec{y})\leq N\Biggl[
\sum \limits_{n\in M_{1}}\frac{1}{2^{n}}\frac{p_{n}(x_{1}-y_{1})}{1+p_{n}(x_{1}-y_{1})}+\cdot \cdot \cdot +\sum \limits_{n\in M_{N}}\frac{1}{2^{n}}\frac{p_{n}(x_{N}-y_{N})}{1+p_{n}(x_{N}-y_{N})}\Biggr]
\\ \leq & N\Biggl[
\sum \limits_{n=1}^{\infty}\frac{1}{2^{n}}\frac{p_{n}(x_{1}-y_{1})}{1+p_{n}(x_{1}-y_{1})}+\cdot \cdot \cdot +\sum \limits_{n=1}^{\infty}\frac{1}{2^{n}}\frac{p_{n}(x_{N}-y_{N})}{1+p_{n}(x_{N}-y_{N})}\Biggr]
\\ = & N\Bigl[ d_{Y}(x_{1},y_{1})+\cdot \cdot \cdot +d_{Y}(x_{N},y_{N})\Bigr]\leq  N^{2}d_{Y^{N}}(\vec{x},\vec{y}),
\end{align*}
as claimed. In the case that $Y$ is a Banach space, then $Y^{N}$ is likewise a Banach space and, in this case, it will be assumed that the distance in $Y^{N}$ is given by $d_{Y^{N}}(\vec{x},\vec{y})=\max_{1\leq j\leq N} \|x_{j}-y_{j}\|_{Y},$ $ \vec{x}\in Y^{N},$ $ \vec{y}\in Y^{N}.$

Let $C\in L(X)$ be injective. Put $p_{n}^{C}(x):=p_{n}(C^{-1}x),$ $n\in {\mathbb N},$ $x\in R(C).$ Then
$p_{n}^{C}(\cdot)$ is a seminorm on $R(C)$ and the calibration
$(p_{n}^{C})_{n\in {\mathbb N}}$ induces a Fr\' echet locally convex topology on
$R(C);$ we denote this space simply by $[R(C)].$ Notice
that $[R(C)]$ is a Banach space (complex Hilbert space) provided that $X$ is.

Set $S_{1}:=\{z\in {\mathbb C} : |z|=1\},$
$\Sigma_{\alpha}:=\{z\in {\mathbb C} : z\neq 0,\ |\arg(z)|<\alpha \}$ ($\alpha \in (0,\pi]$), $\lceil s \rceil:=\inf\{k\in {\mathbb Z} : s\leq k\}$ and
${\mathbb N}_{n}:=\{1,\cdot \cdot \cdot,n\}$ ($s\in {\mathbb R},$ $n\in {\mathbb N}$).
 Let us recall that the upper density of a set $D\subseteq {\mathbb N}$ is defined
by
$$
\overline{dens}(D):=\limsup_{n\rightarrow
+\infty}\frac{card(D \cap [1,n])}{n}.
$$

\section{Multivalued linear operators}\label{MLOs}

In this section, we present a brief overview
of the necessary definitions and properties of multivalued linear operators\index{multivalued linear operator!MLO}. For more details on the subject, we refer the reader to the monographs  \cite{cross} by R. Cross, \cite{faviniyagi} by A. Favini, A. Yagi, and \cite{FKP} by M. Kosti\' c.

Let $X$ and $Y$ be two Fr\' echet spaces over the same field of scalars.
A multivalued map (multimap) ${\mathcal A} : X \rightarrow P(Y)$ is said to be a multivalued
linear operator (MLO) iff the following holds:
\begin{itemize}
\item[(i)] $D({\mathcal A}) := \{x \in X : {\mathcal A}x \neq \emptyset\}$ is a linear subspace of $X$;
\item[(ii)] ${\mathcal A}x +{\mathcal A}y \subseteq {\mathcal A}(x + y),$ $x,\ y \in D({\mathcal A})$
and $\lambda {\mathcal A}x \subseteq {\mathcal A}(\lambda x),$ $\lambda \in {\mathbb K},$ $x \in D({\mathcal A}).$
\end{itemize}
If $X=Y,$ then we say that ${\mathcal A}$ is an MLO in $X.$
An almost immediate consequence of definition is that ${\mathcal A}x +{\mathcal A}y = {\mathcal A}(x+y)$ for
all $ x,\ y \in D({\mathcal A})$ and $\lambda {\mathcal A}x = {\mathcal A}(\lambda x)$ for all $x \in D({\mathcal A}),$ $ \lambda \neq 0.$ Furthermore,
for any $x,\ y\in D({\mathcal A})$ and $\lambda,\ \eta \in {\mathbb K}$ with $|\lambda| + |\eta| \neq 0,$ we
have $\lambda {\mathcal A}x + \eta {\mathcal A}y = {\mathcal A}(\lambda x + \eta y).$ If ${\mathcal A}$ is an MLO, then ${\mathcal A}0$ is a linear manifold in $Y$
and ${\mathcal A}x = f + {\mathcal A}0$ for any $x \in D({\mathcal A})$ and $f \in {\mathcal A}x.$ Set $R({\mathcal A}):=\{{\mathcal A}x :  x\in D({\mathcal A})\}.$
The set ${\mathcal A}^{-1}0 = \{x \in D({\mathcal A}) : 0 \in {\mathcal A}x\}$ is called the kernel\index{multivalued linear operator!kernel}
of ${\mathcal A}$ and it is denoted henceforth by $N({\mathcal A})$ or Kern$({\mathcal A}).$ The inverse ${\mathcal A}^{-1}$ of an MLO is defined by
$D({\mathcal A}^{-1}) := R({\mathcal A})$ and ${\mathcal A}^{-1} y := \{x \in D({\mathcal A}) : y \in {\mathcal A}x\}$.\index{multivalued linear operator!inverse}
It is checked at once that ${\mathcal A}^{-1}$ is an MLO in $X,$ as well as that $N({\mathcal A}^{-1}) = {\mathcal A}0$
and $({\mathcal A}^{-1})^{-1}={\mathcal A}.$ If $N({\mathcal A}) = \{0\},$ i.e., if ${\mathcal A}^{-1}$ is
single-valued, then ${\mathcal A}$ is said to be injective. It is worth noting that ${\mathcal A}x = {\mathcal A}y$ for some two elements $x$ and $y\in D({\mathcal A}),$
iff ${\mathcal A}x \cap {\mathcal A}y \neq \emptyset;$ moreover, if ${\mathcal A}$ is injective, then the equality ${\mathcal A}x = {\mathcal A}y$ holds iff $x = y.$

For any mapping ${\mathcal A}: X \rightarrow P(Y)$ we define $\check{{\mathcal A}}:=\{(x,y) : x\in D({\mathcal A}),\ y\in {\mathcal A}x\}.$ Then ${\mathcal A}$ is an MLO iff $\check{{\mathcal A}}$ is a linear relation in $X\times Y,$ i.e., iff $\check{{\mathcal A}}$ is a linear subspace of $X \times Y.$ In our work, we will occasionally identify ${\mathcal A}$ and its associated linear relation $\check{{\mathcal A}}$.

If ${\mathcal A},\ {\mathcal B} : X \rightarrow P(Y)$ are two MLOs, then we define its sum ${\mathcal A}+{\mathcal B}$ by $D({\mathcal A}+{\mathcal B}) := D({\mathcal A})\cap D({\mathcal B})$ and $({\mathcal A}+{\mathcal B})x := {\mathcal A}x +{\mathcal B}x,$ $x\in D({\mathcal A}+{\mathcal B}).$
It can be simply verified that ${\mathcal A}+{\mathcal B}$ is likewise an MLO. We write ${\mathcal A} \subseteq {\mathcal B}$ iff $D({\mathcal A}) \subseteq D({\mathcal B})$ and ${\mathcal A}x \subseteq {\mathcal B}x$
for all $x\in D({\mathcal A}).$ 
Let us recall that ${\mathcal A}$ is called purely multivalued iff ${\mathcal A}0\neq \{0\}.$ \index{multivalued linear operator!sum}

Let ${\mathcal A} : X \rightarrow P(Y)$ and ${\mathcal B} : Y\rightarrow P(Z)$ be two MLOs, where $Z$ is a Fr\' echet space over the same field of scalars as $X$ and $Y$. The product of ${\mathcal A}$
and ${\mathcal B}$ is defined by $D({\mathcal B}{\mathcal A}) :=\{x \in D({\mathcal A}) : D({\mathcal B})\cap {\mathcal A}x \neq \emptyset\}$ and\index{multivalued linear operator!product}
${\mathcal B}{\mathcal A}x:=
{\mathcal B}(D({\mathcal B})\cap {\mathcal A}x).$ Then ${\mathcal B}{\mathcal A} : X\rightarrow P(Z)$ is an MLO and
$({\mathcal B}{\mathcal A})^{-1} = {\mathcal A}^{-1}{\mathcal B}^{-1}.$ The scalar multiplication of an MLO ${\mathcal A} : X\rightarrow P(Y)$ with the number $z\in {\mathbb K},$ $z{\mathcal A}$ for short, is defined by
$D(z{\mathcal A}):=D({\mathcal A})$ and $(z{\mathcal A})(x):=z{\mathcal A}x,$ $x\in D({\mathcal A}).$ It is clear that $z{\mathcal A}  : X\rightarrow P(Y)$ is an MLO and $(\omega z){\mathcal A}=\omega(z{\mathcal A})=z(\omega {\mathcal A}),$ $z,\ \omega \in {\mathbb K}.$

Suppose that $X'$ is a linear subspace of $X,$ and ${\mathcal A} : X\rightarrow P(Y)$ is an MLO. Then we define the restriction of operator ${\mathcal A}$ to the subspace $X',$ ${\mathcal A}_{|X'}$ for short, by $D({\mathcal A}_{|X'}):=D({\mathcal A}) \cap X'$ and ${\mathcal A}_{|X'}x:={\mathcal A}x,$ $x\in D({\mathcal A}_{|X'}).$\index{multivalued linear operator!restriction}
Clearly, ${\mathcal A}_{|X'} : X\rightarrow P(Y)$ is an MLO.

The integer powers of an MLO ${\mathcal A} :  X\rightarrow P(X)$ are defined recursively as follows: ${\mathcal A}^{0}=:I;$ if ${\mathcal A}^{n-1}$ is defined, set\index{multivalued linear operator!integer powers}
$$
D({\mathcal A}^{n}) := \bigl\{x \in  D({\mathcal A}^{n-1}) : D({\mathcal A}) \cap {\mathcal A}^{n-1}x \neq \emptyset \bigr\},
$$
and
$$
{\mathcal A}^{n}x := \bigl({\mathcal A}{\mathcal A}^{n-1}\bigr)x =\bigcup_{y\in  D({\mathcal A}) \cap {\mathcal A}^{n-1}x}{\mathcal A}y,\quad x\in D( {\mathcal A}^{n}).
$$
We can prove inductively that $({\mathcal A}^{n})^{-1} = ({\mathcal A}^{n-1})^{-1}{\mathcal A}^{-1} = ({\mathcal A}^{-1})^{n}=:{\mathcal A}^{-n},$ $n \in {\mathbb N}$
and $D((\lambda-{\mathcal A})^{n})=D({\mathcal A}^{n}),$ $n \in {\mathbb N}_{0},$ $\lambda \in {\mathbb C}.$ Moreover,
if ${\mathcal A}$ is single-valued, then the above definition is consistent with the usual definition of powers of ${\mathcal A}.$ Set $D_{\infty}({\mathcal A}):=\bigcap_{n\in {\mathbb N}}D({\mathcal A}^{n}).$

Suppose that ${\mathcal A}$ is an MLO in $ X.$ Then we say that a point $\lambda \in {\mathbb C}$ is an eigenvalue of ${\mathcal A}$
iff there exists a vector $x\in X\setminus \{0\}$ such that $\lambda x\in {\mathcal A}x;$ we call $x$ an eigenvector of operator ${\mathcal A}$ corresponding to the eigenvalue $\lambda.$ Observe that, in purely multivalued case, a vector $x\in X\setminus \{0\}$ can be an  eigenvector of operator ${\mathcal A}$ corresponding to different values of scalars $\lambda.$ The point spectrum of ${\mathcal A},$ $\sigma_{p}({\mathcal A})$ for short, is defined as the union of all eigenvalues of ${\mathcal A}.$


In our work, the important role has the multivalued linear operator $B^{-1}A,$ where $A$ and $B$ are single-valued linear operators acting between the spaces $X$ and $Y,$
and $B$ is not necessarily injective. Then $B^{-1}A=\{( x,y) \in X \times X : Ax=By\}$ is an MLO in $X.$ 

\subsection{Distributional chaos for MLOs}\label{marek-MLOs}

In the following definition, we recall the notion of $\tilde{X}$-distributional chaos for MLOs and their sequences (cf. \cite{marek-trio}, and \cite[Definition 3.1]{mendoza} for single-valued linear case).

\begin{defn}\label{DC-unbounded-fric} 
Suppose that, for every $k\in {\mathbb N},$ ${\mathcal A}_{k} : D({\mathcal A}_{k}) \subseteq X \rightarrow Y$ is an MLO and $\tilde{X}$ is a closed linear
subspace of $X.$ Then we say that the sequence $({\mathcal A}_{k})_{k\in
{\mathbb N}}$ is
$\tilde{X}$-distributionally chaotic\index{$\tilde{X}$-distributionally chaotic sequence of operators} iff there exist an uncountable
set $S\subseteq \bigcap_{k=1}^{\infty} D({\mathcal A}_{k}) \cap \tilde{X}$ and
$\sigma>0$ such that for each $\epsilon>0$ and for each pair $x,\
y\in S$ of distinct points we have that for each $k\in {\mathbb N}$ there exist elements $x_{k}\in {\mathcal A}_{k}x$ and $y_{k} \in {\mathcal A}_{k}y$ such that the following equalities hold:
\begin{align*}
\begin{split}
& \overline{dens}\Bigl( \bigl\{k \in {\mathbb N} :
d_{Y}\bigl(x_{k},y_{k}\bigr)\geq \sigma \bigr\}\Bigr)=1,
\\
& \overline{dens}\Bigl(\bigl\{k \in {\mathbb N} : d_{Y}\bigl(x_{k},y_{k}\bigr)
<\epsilon \bigr\}\Bigr)=1.
\end{split}
\end{align*}
The sequence $({\mathcal A}_{k})_{k\in
{\mathbb N}}$ is said to be densely\index{$\tilde{X}$-distributionally chaotic sequence of operators!densely}
$\tilde{X}$-distributionally chaotic iff $S$ can be chosen to be dense in $\tilde{X}.$
An MLO ${\mathcal A} : D({\mathcal A})\subseteq X \rightarrow X$ is said to be (densely)
$\tilde{X}$-distributionally chaotic\index{$\tilde{X}$-distributionally chaotic operator} iff the sequence $({\mathcal A}_{k}\equiv
{\mathcal A}^{k})_{k\in {\mathbb N}}$ is.
The set $S$ is said to be $\sigma_{\tilde{X}}$-scrambled set\index{ $\sigma_{\tilde{X}}$-scrambled set} ($\sigma$-scrambled set in the case\index{$\sigma$-scrambled set}
that $\tilde{X}=X$) of the sequence $({\mathcal A}_{k})_{k\in {\mathbb N}}$ (the
operator ${\mathcal A}$);  in the case that
$\tilde{X}=X,$ then we also say that the sequence $({\mathcal A}_{k})_{k\in
{\mathbb N}}$ (the operator ${\mathcal A}$) is distributionally chaotic.
\end{defn}

In the following definition appearing in \cite{marek-trio}, we have recently adapted the notion introduced in \cite[Definition 3.4]{mendoza} for multivalued linear operators.

\begin{defn}\label{DC-unbounded-fric-prim} 
Suppose that, for every $k\in {\mathbb N},$ ${\mathcal A}_{k} : D({\mathcal A}_{k})\subseteq X \rightarrow Y$ is an MLO, $\tilde{X}$ is a closed linear
subspace of $X,$ $x\in \bigcap_{k=1}^{\infty}D({\mathcal A}_{k})$ and
$m\in {\mathbb N}.$ Then we say that:
\begin{itemize}
\item[(i)] $x$ is distributionally
near to $0$ for $({\mathcal A}_{k})_{k\in {\mathbb N}}$
iff there exists $A\subseteq {\mathbb N}$ such that
$\overline{dens}(A)=1$ and for each $k\in A$ there exists $x_{k}\in {\mathcal A}_{k}x$ such that
$\lim_{k\in A,k\rightarrow \infty}x_{k}=0;$
\item[(ii)] $x$ is distributionally $m$-unbounded for $({\mathcal A}_{k})_{k\in {\mathbb N}}$ iff there exists $B\subseteq
{\mathbb N}$ such that $\overline{dens}(B)=1$ and for each
$k\in B$ there exists $x_{k}\in {\mathcal A}_{k}x$ such that
$\lim_{k\in
B,k\rightarrow \infty}p_{m}^{Y}(x_{k})=\infty ;$ $x$ is said to be distributionally unbounded for $({\mathcal A}_{k})_{k\in {\mathbb N}}$ iff there
exists $q\in {\mathbb N}$ such that $x$ is distributionally $q$-unbounded for $({\mathcal A}_{k})_{k\in {\mathbb N}}$
(if $Y$ is a Banach space, this means $\lim_{k\in
B,k\rightarrow \infty}\|x_{k}\|_{Y}=\infty );$
\item[(iii)] $x$ is a $\tilde{X}$-distributionally irregular vector\index{$\tilde{X}$-distributionally irregular vector} for
$({\mathcal A}_{k})_{k\in {\mathbb N}}$  (distributionally
irregular vector for $({\mathcal A}_{k})_{k\in {\mathbb N}}$, in
the case that $\tilde{X}=X$) iff $x\in
\bigcap_{k=1}^{\infty}D({\mathcal A}_{k}) \cap \tilde{X}$ and (i)-(ii) hold.
\end{itemize}
If ${\mathcal A} : D({\mathcal A})\subseteq X \rightarrow Y$ is an MLO and $x\in
D_{\infty}({\mathcal A}),$ then we say that
$x$ is distributionally
near to $0$ (distributionally $m$-unbounded, distributionally unbounded) for ${\mathcal A}$ iff
$x$ is distributionally
near to $0$ (distributionally $m$-unbounded, distributionally unbounded) for the sequence
$({\mathcal A}_{k}\equiv {\mathcal A}^{k})_{k\in {\mathbb N}};$
$x$ is said
to be a $\tilde{X}$-distributionally irregular vector for
${\mathcal A}$ (distributionally irregular vector for ${\mathcal A},$ in the case
that $\tilde{X}=X$) iff $x$ is $\tilde{X}$-distributionally
irregular vector for the sequence $({\mathcal A}_{k}\equiv {\mathcal A}^{k})_{k\in {\mathbb N}}$
(distributionally irregular vector\index{distributionally irregular vector} for
$({\mathcal A}_{k}\equiv {\mathcal A}^{k})_{k\in {\mathbb N}},$ in the case that $\tilde{X}=X$).
\end{defn}

\section{Disjoint distributionally chaotic properties of MLOs}\label{d-MLOs}

Let $\sigma>0,$ let $ \epsilon>0,$ and let $(x_{j,k})_{k\in {\mathbb N}}$ and $(y_{j,k})_{k\in {\mathbb N}}$ be sequences in $X$ ($1\leq j\leq N$). Consider the following conditions:
\begin{align}\label{d-jednacina}
\begin{split}
& \overline{dens}\Biggl( \bigcap_{j\in {\mathbb N}_{N}} \bigl\{k \in {\mathbb N} :
d_{Y}\bigl(x_{j,k},y_{j,k}\bigr)\geq \sigma \bigr\}\Biggr)=1,\mbox{ and }
\\ 
& \overline{dens}\Biggl( \bigcap_{j\in {\mathbb N}_{N}} \bigl\{k \in {\mathbb N} : d_{Y}\bigl(x_{j,k},y_{j,k}\bigr)
<\epsilon \bigr\}\Biggr)=1;
\end{split}
\end{align}
\begin{align}\label{d-jednacina2}
\begin{split}
& \overline{dens}\Biggl( \bigcap_{j\in {\mathbb N}_{N}} \bigl\{k \in {\mathbb N} :
d_{Y}\bigl(x_{j,k},y_{j,k}\bigr)\geq \sigma \bigr\}\Biggr)=1,\mbox{ and }
\\ 
& \bigl(\forall j\in {\mathbb N}_{N}\bigr)\, \ \overline{dens}\Bigl(  \bigl\{k \in {\mathbb N} : d_{Y}\bigl(x_{j,k},y_{j,k}\bigr)
<\epsilon \bigr\}\Bigr)=1;
\end{split}
\end{align}
\begin{align}\label{d-jednacina3}
\begin{split}
& \bigl(\forall j\in {\mathbb N}_{N}\bigr)\, \ \overline{dens}\Bigl( \bigl\{k \in {\mathbb N} :
d_{Y}\bigl(x_{j,k},y_{j,k}\bigr)\geq \sigma \bigr\}\Bigr)=1,\mbox{ and }
\\ 
& \bigl(\forall j\in {\mathbb N}_{N}\bigr)\, \ \overline{dens}\Bigl(  \bigl\{k \in {\mathbb N} : d_{Y}\bigl(x_{j,k},y_{j,k}\bigr)
<\epsilon \bigr\}\Bigr)=1;
\end{split}
\end{align}
\begin{align}\label{d-jednacina3'}
\begin{split}
& \bigl(\forall j\in {\mathbb N}_{N}\bigr)\, \ \overline{dens}\Bigl( \bigl\{k \in {\mathbb N} :
d_{Y}\bigl(x_{j,k},y_{j,k}\bigr)\geq \sigma \bigr\}\Bigr)=1,\mbox{ and }
\\ 
& \bigl(\exists j\in {\mathbb N}_{N}\bigr)\, \ \overline{dens}\Bigl(  \bigl\{k \in {\mathbb N} : d_{Y}\bigl(x_{j,k},y_{j,k}\bigr)
<\epsilon \bigr\}\Bigr)=1;
\end{split}
\end{align}
\begin{align}\label{d-jednacina3''}
\begin{split}
& \bigl(\exists j\in {\mathbb N}_{N}\bigr)\, \ \overline{dens}\Bigl( \bigl\{k \in {\mathbb N} :
d_{Y}\bigl(x_{j,k},y_{j,k}\bigr)\geq \sigma \bigr\}\Bigr)=1,\mbox{ and }
\\ 
& \bigl(\forall j\in {\mathbb N}_{N}\bigr)\, \ \overline{dens}\Bigl(  \bigl\{k \in {\mathbb N} : d_{Y}\bigl(x_{j,k},y_{j,k}\bigr)
<\epsilon \bigr\}\Bigr)=1;
\end{split}
\end{align}
\begin{align}\label{d-jednacina4}
\begin{split}
& \overline{dens}\Biggl( \bigcap_{j\in {\mathbb N}_{N}} \bigl\{k \in {\mathbb N} :
d_{Y}\bigl(x_{j,k},y_{j,k}\bigr)\geq \sigma \bigr\}\Biggr)=1,\mbox{ and }
\\ 
& \bigl(\exists j\in {\mathbb N}_{N}\bigr)\, \ \overline{dens}\Bigl(  \bigl\{k \in {\mathbb N} : d_{Y}\bigl(x_{j,k},y_{j,k}\bigr)
<\epsilon \bigr\}\Bigr)=1;
\end{split}
\end{align}
\begin{align}\label{d-jednacina7}
\begin{split}
& \bigl(\forall j\in {\mathbb N}_{N}\bigr)\, \ \overline{dens}\Bigl(  \bigl\{k \in {\mathbb N} :
d_{Y}\bigl(x_{j,k},y_{j,k}\bigr)\geq \sigma \bigr\}\Bigr)=1,\mbox{ and }
\\ 
& \overline{dens}\Biggl( \bigcap_{j\in {\mathbb N}_{N}} \bigl\{k \in {\mathbb N} : d_{Y}\bigl(x_{j,k},y_{j,k}\bigr)
<\epsilon \bigr\}\Biggr)=1;
\end{split}
\end{align}
\begin{align}\label{d-jednacina8}
\begin{split}
& \bigl(\exists j\in {\mathbb N}_{N}\bigr)\, \ \overline{dens}\Bigl(  \bigl\{k \in {\mathbb N} :
d_{Y}\bigl(x_{j,k},y_{j,k}\bigr)\geq \sigma \bigr\}\Bigr)=1,\mbox{ and }
\\ 
& \overline{dens}\Biggl( \bigcap_{j\in {\mathbb N}_{N}} \bigl\{k \in {\mathbb N} : d_{Y}\bigl(x_{j,k},y_{j,k}\bigr)
<\epsilon \bigr\}\Biggr)=1;
\end{split}
\end{align}
\begin{align}\label{d-jednacina9}
\begin{split}
& \overline{dens}\Biggl( \bigcup_{j\in {\mathbb N}_{N}}  \bigl\{k \in {\mathbb N} :
d_{Y}\bigl(x_{j,k},y_{j,k}\bigr)\geq \sigma \bigr\}\Biggr)=1,\mbox{ and }
\\ 
& \overline{dens}\Biggl( \bigcap_{j\in {\mathbb N}_{N}} \bigl\{k \in {\mathbb N} : d_{Y}\bigl(x_{j,k},y_{j,k}\bigr)
<\epsilon \bigr\}\Biggr)=1;
\end{split}
\end{align}
\begin{align}\label{d-jednacina10}
\begin{split}
& \overline{dens}\Biggl( \bigcup_{j\in {\mathbb N}_{N}}  \bigl\{k \in {\mathbb N} :
d_{Y}\bigl(x_{j,k},y_{j,k}\bigr)\geq \sigma \bigr\}\Biggr)=1,\mbox{ and }
\\ 
& \bigl(\forall j\in {\mathbb N}_{N}\bigr)\, \ \overline{dens}\Bigl(  \bigl\{k \in {\mathbb N} : d_{Y}\bigl(x_{j,k},y_{j,k}\bigr)
<\epsilon \bigr\}\Bigr)=1;
\end{split}
\end{align}
\begin{align}\label{d-jednacina11}
\begin{split}
&\overline{dens}\Biggl( \bigcap_{j\in {\mathbb N}_{N}} \bigl\{k \in {\mathbb N} :
d_{Y}\bigl(x_{j,k},y_{j,k}\bigr)\geq \sigma \bigr\}\Biggr)=1,\mbox{ and }
\\ 
& \overline{dens}\Biggl( \bigcup_{j\in {\mathbb N}_{N}} \bigl\{k \in {\mathbb N} : d_{Y}\bigl(x_{j,k},y_{j,k}\bigr)
<\epsilon \bigr\}\Biggr)=1;
\end{split}
\end{align}
\begin{align}\label{d-jednacina12}
\begin{split}
& \bigl(\forall j\in {\mathbb N}_{N}\bigr)\, \ \overline{dens}\Bigl(  \bigl\{k \in {\mathbb N} :
d_{Y}\bigl(x_{j,k},y_{j,k}\bigr)\geq \sigma \bigr\}\Bigr)=1,\mbox{ and }
\\ 
& \overline{dens}\Biggl( \bigcup_{j\in {\mathbb N}_{N}} \bigl\{k \in {\mathbb N} : d_{Y}\bigl(x_{j,k},y_{j,k}\bigr)
<\epsilon \bigr\}\Biggr)=1.
\end{split}
\end{align}

Now we are ready to introduce the following notion of disjoint distributional chaos for MLOs in Fr\'echet spaces:

\begin{defn}\label{DC-unbounded-fric-DISJOINT}
Let $i\in {\mathbb N}_{12}.$
Suppose that, for every $j\in {\mathbb N}_{N}$ and $k\in {\mathbb N},$ ${\mathcal A}_{j,k} : D({\mathcal A}_{j,k})\subseteq X \rightarrow Y$ is an MLO and $\tilde{X}$ is a closed linear
subspace of $X.$
Then we say that the sequence $(({\mathcal A}_{j,k})_{k\in
{\mathbb N}})_{1\leq j\leq N}$ is disjoint
$(\tilde{X},i)$-distributionally chaotic, $(d,\tilde{X},i)$-distributionally chaotic in short, iff there exist an uncountable
set $S\subseteq \bigcap_{j=1}^{N} \bigcap_{k=1}^{\infty} D({\mathcal A}_{j,k}) \cap \tilde{X}$ and
$\sigma>0$ such that for each $\epsilon>0$ and for each pair $x,\
y\in S$ of distinct points we have that for each $j\in {\mathbb N}_{N}$ and $k\in {\mathbb N}$ there exist elements $x_{j,k}\in {\mathcal A}_{j,k}x$ and $y_{j,k}\in {\mathcal A}_{j,k}y$ such that
(3.i) holds.

The sequence $(({\mathcal A}_{j,k})_{k\in
{\mathbb N}})_{1\leq j\leq N}$ is said to be densely
$(d,\tilde{X},i)$-distributionally chaotic iff $S$ can be chosen to be dense in $\tilde{X}.$
A finite sequence $({\mathcal A}_{j})_{1\leq j\leq N}$ of MLOs on $X$ is said to be (densely)
$(\tilde{X},i)$-distributionally chaotic\index{$\tilde{X}$-distributionally chaotic operator} iff the sequence $(({\mathcal A}_{j,k}\equiv
{\mathcal A}_{j}^{k})_{k\in {\mathbb N}})_{1\leq j\leq N}$ is.
The set $S$ is said to be $(d,\sigma_{\tilde{X}},i)$-scrambled set\index{ $\sigma_{\tilde{X}}$-scrambled set} ($(d,\sigma,i)$-scrambled set in the case\index{$\sigma$-scrambled set}
that $\tilde{X}=X$) of $(({\mathcal A}_{j,k})_{k\in
{\mathbb N}})_{1\leq j\leq N}$ ($({\mathcal A}_{j})_{1\leq j\leq N}$);  in the case that
$\tilde{X}=X,$ then we also say that the sequence $(({\mathcal A}_{j,k})_{k\in
{\mathbb N}})_{1\leq j\leq N}$ ($({\mathcal A}_{j})_{1\leq j\leq N}$) is disjoint $i$-distributionally chaotic, $(d,i)$-distributionally chaotic in short.
\end{defn}

Concerning Definition \ref{DC-unbounded-fric-DISJOINT}, it should be noted that the use of any strongly equivalent metric $d_{Y}'(\cdot,\cdot)$ with $d_{Y}(\cdot,\cdot)$ in (3.i) leads to the same notion of distributional chaos. The use of manifold $\tilde{X}\neq X$ in linear topological dynamics comes from the paper \cite{banasiak} by 
J. Banasiak and M. Moszy\'nski; as already marked in \cite{mendoza}, we need to know the minimal linear subspace $\tilde{X}$ of $X,$ in a certain sense, for which the sequence $(({\mathcal A}_{j,k})_{k\in
{\mathbb N}})_{1\leq j\leq N}$ is 
$(d,\tilde{X},i)$-distributionally chaotic because, in this case, $(({\mathcal A}_{j,k})_{k\in
{\mathbb N}})_{1\leq j\leq N}$ is 
$(d,\tilde{X}',i)$-distributionally chaotic for any closed linear subspace $\tilde{X}'$ of $X$ containing $\tilde{X}.$

\begin{example}\label{totan}
It is expected that the multivalued linear operators $X\times X,\cdot \cdot \cdot,\, X\times X,$ totally counted $N$ times, are densely $(d,1)$-distributionally chaotic. To see this, take any 
two disjoint subsets $A$ and $B$ of ${\mathbb N}$ such that ${\mathbb N}=A\cup B$ and $\overline{dens}(A)=\overline{dens}(B)=1.$ If $k\in A$ and $j\in {\mathbb N}_{N},$ choose any $z_{j,k}:=x_{j,k}-y_{j,k}$ such that $d_{Y}(z_{j,k},0)\geq 1;$  if $k\in B$ and $j\in {\mathbb N}_{N},$ choose simply $x_{j,k}=y_{j,k}.$ Then it is easy to see that (3.1) holds with $S=X\times X$ and $\sigma=1.$ 
\end{example}

In particular, the previous example shows that dense (full, moreover) $(d,1)$-distributional chaos occurs in finite-dimensional spaces for the sequences of MLOs. And, more to the point, the full $(d,1)$-distributional chaos occurs in finite-dimensional spaces even for the sequences of linear continuous operators, as the following example shows (this fact has not been observed in \cite{2013JFA} and \cite{mendoza}):

\begin{example}\label{totanr}
Let $X:={\mathbb K}^{n},$ and let $A$ and $B$ be 
two disjoint subsets of ${\mathbb N}$ such that ${\mathbb N}=A\cup B$ and $\overline{dens}(A)=\overline{dens}(B)=1.$ If $k\in A$ and $j\in {\mathbb N}_{N},$ we choose $T_{j,k}$ to be the diagonal matrix $diag(a_{11}^{jk},\cdot \cdot \cdot,a_{nn}^{jk})$ such that $|a_{ii}^{jk}|\geq j+k$ for all $i\in {\mathbb N}_{n}.$ If $k\in B$ and $j\in {\mathbb N}_{N},$ we define $T_{j,k}:=0.$ Then it can be simply verified that (3.1) holds with $S=X$ and $\sigma>0$ arbitrarily chosen. 
\end{example}

On the other hand, for each integer $i\in {\mathbb N}_{8}$ we have that the $(d,\tilde{X},i)$-distributional chaos of operators $T_{1}\in L(X),\cdot \cdot \cdot, T_{N}\in L(X)$ implies that there exists an index $j\in {\mathbb N}_{N}$ such that
$T_{j}$ is $(\tilde{X},i)$-distributionally chaotic and therefore both Li-Yorke chaotic and distributionally chaotic (see also Remark \ref{prc-qwe} below and \cite[Definition 1, Theorem 5]{2011}). If this is the case, $T_{j}$ cannot be a compact operator due to \cite[Corollary 6]{2011} and, because of that, $X$ needs to be infinite-dimensional in this case. The same holds if $i\in \{9,10,11,12\}$
because then the $(d,\tilde{X},i)$-distributional chaos of operators $T_{1}\in L(X),\cdot \cdot \cdot, T_{N}\in L(X)$ implies that there exists an index $j\in {\mathbb N}_{N}$ such that
$T_{j}$ is Li-Yorke chaotic.

The following important proposition can be trivially deduced (the parts [10. and 12.] will be specified a little bit later, in Example \ref{10} and Example \ref{12}):

\begin{prop}\label{trick}
For any sequence ${\mathbb A}\equiv (({\mathcal A}_{j,k})_{k\in
{\mathbb N}})_{1\leq j\leq N}$ of \emph{MLOs}, the following holds:
\begin{itemize}
\item[1.] $(d,\tilde{X},1)$-distributional chaos of ${\mathbb A}$ implies $(d,\tilde{X},i)$-distributional chaos of ${\mathbb A}$ for all $i\in {\mathbb N}_{12};$
\item[2.] $(d,\tilde{X},2)$-distributional chaos implies $(d,\tilde{X},i)$-distributional chaos for all $i\in\{3,4,5,6,10,11,12\};$
\item[3.] $(d,\tilde{X},3)$-distributional chaos of ${\mathbb A}$ implies $(d,\tilde{X},i)$-distributional chaos of ${\mathbb A}$ for all $i\in\{4,5,10,12\};$
\item[4.] $(d,\tilde{X},4)$-distributional chaos of ${\mathbb A}$ implies $(d,\tilde{X},12)$-distributional chaos of ${\mathbb A}$;
\item[5.] $(d,\tilde{X},5)$-distributional chaos of ${\mathbb A}$ implies $(d,\tilde{X},10)$-distributional chaos of ${\mathbb A}$;
\item[6.] $(d,\tilde{X},6)$-distributional chaos of ${\mathbb A}$ implies $(d,\tilde{X},i)$-distributional chaos of ${\mathbb A}$ for all $i\in\{4,11,12\};$
\item[7.] $(d,\tilde{X},7)$-distributional chaos of ${\mathbb A}$ implies $(d,\tilde{X},i)$-distributional chaos of ${\mathbb A}$ for all $i\in\{3,4,5,8,9,10,12\};$
\item[8.] $(d,\tilde{X},8)$-distributional chaos of ${\mathbb A}$ implies $(d,\tilde{X},i)$-distributional chaos of ${\mathbb A}$ for all $i\in\{5,9,10\};$
\item[9.] $(d,\tilde{X},9)$-distributional chaos of ${\mathbb A}$ implies $(d,\tilde{X},10)$-distributional chaos of ${\mathbb A}$;
\item[10.] $(d,\tilde{X},10)$-distributional chaos of ${\mathbb A}$ does not imply anything, in general;
\item[11.] $(d,\tilde{X},11)$-distributional chaos of ${\mathbb A}$ implies $(d,\tilde{X},12)$-distributional chaos of ${\mathbb A}$;
\item[12.] $(d,\tilde{X},12)$-distributional chaos of ${\mathbb A}$ does not imply anything, in general.
\end{itemize}
\end{prop}

In view of this, $(d,\tilde{X},1)$-distributional chaos is unquestionably the most important, because it implies all others. But, in our analyses, 
the notion of $(d,\tilde{X},9)$-distributional chaos is incredibly important, as well (see also \cite[Definition 2.2, Remark 2.9]{bp07}):

\begin{prop}\label{tuple-profo}
Suppose that, for every $j\in {\mathbb N}_{N}$ and $k\in {\mathbb N},$ ${\mathcal A}_{j,k} : D({\mathcal A}_{k}) \subseteq X \rightarrow Y$ is an \emph{MLO} and $\tilde{X}$ is a closed linear
subspace of $X.$ 
Define, for every $k\in {\mathbb N},$ the \emph{MLO} ${\mathbb A}_{k} : D({\mathbb A}_{k}) \subseteq X \rightarrow Y^{N}$
by
$D({\mathbb A}_{k}):=\bigcap_{1\leq j\leq N}D({\mathcal A}_{j,k})$
and ${\mathbb A}_{k}x:=\{(x_{1,k},\cdot \cdot \cdot,x_{N,k}) : x_{j,k} \in {\mathcal A}_{j,k}x\mbox{ for all }j\in {\mathbb N}_{N}\}.$
Then the sequence $(({\mathcal A}_{j,k})_{k\in
{\mathbb N}})_{1\leq j\leq N}$ is
disjoint
$(\tilde{X},9)$-distributionally chaotic iff the sequence $({\mathbb A}_{k})_{k\in
{\mathbb N}}$ is
$\tilde{X}$-distributionally chaotic.
\end{prop}

\begin{proof}
The proof of proposition almost trivially follows by elementary definitions and the properties of metric $d_{Y^{N}}(\cdot,\cdot).$
\end{proof}

Based on our considerations, we can introduce and study a great number of other types of disjoint distributional chaos for multivalued linear operators. For example,
we can introduce the notion in which the tuple $(({\mathcal A}_{j,k,r})_{k\in
{\mathbb N}})_{1\leq j\leq N}$
is $(d,\tilde{X},i)$-distributionally chaotic, where $i\in {\mathbb N}_{12},$ $r\in {\mathbb N}$ is given in advance and ${\mathcal A}_{j,k,r}$ denotes the direct sum of $r$ operators ${\mathcal A}_{j,k}$ ($k\in {\mathbb N},$ $1\leq j\leq N$); see \cite{kerry-drew} for the notion and some results in this direction. Because of space and time limitations, we will not follow this approach here.

In our framework, we use the set operations $\cap,$ $\cup$ as well as the quantifiers $\forall$ and $\exists.$ Depending on their choice, we recognize four different types of disjoint distributionally unbounded vectors and 
four different types of disjoint distributionally near to $0$ vectors (in multivalued linear setting, the zero vector can be also disjoint distributionally near to $0$ or disjoint distributionally unbounded but we will not consider this option for the sake of brevity):

\begin{defn}\label{DC-zero-DISJOINT}
Suppose that, for every $j\in {\mathbb N}_{N}$ and $k\in {\mathbb N},$ ${\mathcal A}_{j,k} : D({\mathcal A}_{j,k})\subseteq X \rightarrow Y$ is an MLO and $x\in \bigcap_{j=1}^{N}\bigcap_{k=1}^{\infty}D({\mathcal A}_{j,k}),$ $x\neq 0.$ Then we say that:
\begin{itemize}
\item[(i)] $x$ is $d$-distributionally
near to $0$ of type $1$ for $(({\mathcal A}_{j,k})_{k\in {\mathbb N}})_{1\leq j\leq N}$
iff there exists $A\subseteq {\mathbb N}$ such that
$\overline{dens}(A)=1$ as well as for each $j\in {\mathbb N}_{N}$ and $k\in {\mathbb N}$ there exists $x_{j,k}\in {\mathcal A}_{j,k}x$ such that
$\lim_{k\in A,k\rightarrow \infty}x_{j,k}=0,$ $j\in {\mathbb N}_{N}$ (the use of $\cap$ in (3.i));
\item[(ii)] $x$ is $d$-distributionally
near to $0$ of type $2$ for $(({\mathcal A}_{j,k})_{k\in {\mathbb N}})_{1\leq j\leq N}$
iff for each $\epsilon>0$, $j\in {\mathbb N}_{N}$ and $k\in {\mathbb N}$ there exists $x_{j,k}\in {\mathcal A}_{j,k}x$ such that the set 
$\bigcup_{j\in {\mathbb N}_{N}} \{k \in {\mathbb N} :
d_{Y}(x_{j,k},y_{j,k})< \epsilon \}$ has the upper density $1$ (the use of $\cup$ in (3.i));
\item[(iii)] $x$ is $d$-distributionally
near to $0$ of type $3$ for $(({\mathcal A}_{j,k})_{k\in {\mathbb N}})_{1\leq j\leq N}$
iff for every $j\in {\mathbb N}_{N}$ there exists a set $A_{j}\subseteq {\mathbb N}$ such that
$\overline{dens}(A_{j})=1$ as well as for each $k\in A_{j}$ there exists $x_{j,k}\in {\mathcal A}_{j,k}x$ such that
$\lim_{k\in A_{j},k\rightarrow \infty}x_{j,k}=0$
(the use of $\forall$ in (3.i));
\item[(iv)] $x$ is $d$-distributionally
near to $0$ of type $4$ for $(({\mathcal A}_{j,k})_{k\in {\mathbb N}})_{1\leq j\leq N}$
iff there exist an integer $j\in {\mathbb N}_{N}$ and a set $A_{j}\subseteq {\mathbb N}$ such that
$\overline{dens}(A_{j})=1$ as well as for each $k\in A_{j}$ there exists $x_{j,k}\in {\mathcal A}_{j,k}x$ such that
$\lim_{k\in A_{j},k\rightarrow \infty}x_{j,k}=0$
(the use of $\exists$ in (3.i)).
\end{itemize}
\end{defn}

\begin{defn}\label{DC-unbounded-DISJOINT}
Suppose that, for every $j\in {\mathbb N}_{N}$ and $k\in {\mathbb N},$ ${\mathcal A}_{j,k} : D({\mathcal A}_{j,k})\subseteq X \rightarrow Y$ is an MLO, $x\in \bigcap_{j=1}^{N}\bigcap_{k=1}^{\infty}D({\mathcal A}_{j,k}),$ $x\neq 0,$ $i\in {\mathbb N}_{4}$ and
$m\in {\mathbb N}.$ Then we say that:
\begin{itemize}
\item[(i)] $x$ is $d$-distributionally $m$-unbounded of type $1$ for $(({\mathcal A}_{j,k})_{k\in {\mathbb N}})_{1\leq j\leq N}$  iff there exists $B\subseteq
{\mathbb N}$ such that $\overline{dens}(B)=1$ as well as for each $j\in {\mathbb N}_{N}$ and
$k\in B$ there exists $x_{j,k}'\in {\mathcal A}_{j,k}x$ such that
$\lim_{k\in
B,k\rightarrow \infty}p_{m}^{Y}(x_{j,k}')=\infty ,$ $j\in {\mathbb N}_{N}$ (the use of $\cap$ in (3.i));
\item[(ii)] $x$ is $d$-distributionally $m$-unbounded of type $2$ for $(({\mathcal A}_{j,k})_{k\in {\mathbb N}})_{1\leq j\leq N}$  iff there exists $B\subseteq
{\mathbb N}$ such that $\overline{dens}(B)=1$ as well as for each $j\in {\mathbb N}_{N}$ and
$k\in B$ there exists $x_{j,k}'\in {\mathcal A}_{j,k}x$ such that
$\lim_{k\in
B,k\rightarrow \infty}\sum_{j\in {\mathbb N}_{N}}p_{m}^{Y}(x_{j,k}')=\infty $ (the use of $\cup$ in (3.i));
\item[(iii)] $x$ is $d$-distributionally $m$-unbounded of type $3$ for $(({\mathcal A}_{j,k})_{k\in {\mathbb N}})_{1\leq j\leq N}$  iff for every $j\in {\mathbb N}_{N}$ there exists $B_{j}\subseteq
{\mathbb N}$ such that $\overline{dens}(B_{j})=1$ as well as for each
$k\in B_{j}$ there exists $x_{j,k}'\in {\mathcal A}_{j,k}x$ such that
$\lim_{k\in
B_{j},k\rightarrow \infty}p_{m}^{Y}(x_{j,k}')=\infty $ (the use of $\forall$ in (3.i));
\item[(iv)] $x$ is $d$-distributionally $m$-unbounded of type $4$ for $(({\mathcal A}_{j,k})_{k\in {\mathbb N}})_{1\leq j\leq N}$  iff there exist an integer $j\in {\mathbb N}_{N} $ and a set $B_{j}\subseteq
{\mathbb N}$ such that $\overline{dens}(B_{j})=1$ as well as for each
$k\in B_{j}$ there exists $x_{j,k}'\in {\mathcal A}_{j,k}x$ such that
$\lim_{k\in
B_{j},k\rightarrow \infty}p_{m}^{Y}(x_{j,k}')=\infty $ (the use of $\exists$ in (3.i)).
\end{itemize}
It is said that $x$ is $d$-distributionally unbounded of type $i$ for $(({\mathcal A}_{j,k})_{k\in {\mathbb N}})_{1\leq j\leq N}$ iff there
exists $q\in {\mathbb N}$ such that $x$ is $d$-distributionally $q$-unbounded of type $i$ for $(({\mathcal A}_{j,k})_{k\in {\mathbb N}})_{1\leq j\leq N}.$
\end{defn}

It is clear that a vector $x$ is $d$-distributionally
near to $0$ of type $3,$ resp. $4,$ for $(({\mathcal A}_{j,k})_{k\in {\mathbb N}})_{1\leq j\leq N}$ iff for each $j\in {\mathbb N}_{N},$ resp. there exists $j\in {\mathbb N}_{N},$ such that $x$ is distributionally
near to $0$ for $({\mathcal A}_{j,k})_{k\in {\mathbb N}};$ a similar statement holds for $d$-distributional $m$-unboundedness of type $3$, resp. $4.$ 
Using this and the fact that the assertions of \cite[Proposition 7, Proposition 9]{2013JFA} hold for general sequences of linear continuous operators (cf. also \cite[Theorem 3.7]{mendoza}), we can immediately clarify several results concerning the existence of 
$d$-distributionally
near to $0$ and $d$-distributionally $m$-unbounded vectors of type $3,$ resp. $4.$

For every type od $\tilde{X}$-disjoint distributional chaos, we can introduce the notion of corresponding $(d,\tilde{X})$-distributionally irregular vectors, as it has been done for the usually examined disjoint hypercyclicity (\cite{bp07}). The things are pretty clear 
and definition goes as follows:

\begin{defn}\label{iregularni-prc}
Suppose that, for every $j\in {\mathbb N}_{N}$ and $k\in {\mathbb N},$ ${\mathcal A}_{j,k} : D({\mathcal A}_{j,k})\subseteq X \rightarrow Y$ is an MLO and $x\in \tilde{X} \cap \bigcap_{j=1}^{N}\bigcap_{k=1}^{\infty}D({\mathcal A}_{j,k}),$ $x\neq 0.$ 
Then we say that:
\begin{itemize}
\item[(i)] $x$ is a $(d,\tilde{X},1)$-distributionally irregular vector\index{$\tilde{X}$-distributionally irregular vector} for
$(({\mathcal A}_{j,k})_{k\in {\mathbb N}})_{1\leq j\leq N}$ iff $x$ is $d$-distributionally
near to $0$ of type $1$ for $(({\mathcal A}_{j,k})_{k\in {\mathbb N}})_{1\leq j\leq N},$ $x$ is $d$-distributionally unbounded of type $1$ for $(({\mathcal A}_{j,k})_{k\in {\mathbb N}})_{1\leq j\leq N}$ and
the requirements of the last condition holds with $x_{j,k}'=x_{j,k},$ i.e., the sequences in definitions of $d$-distributionally nearness to $0$ of type $1$ and $d$-distributionally unboundedness of type $1$ 
must be the same ({\it for the sake of brevity, in all remaining parts of this definition, we will assume a priori this condition}; the use of $\cap \cap$ in (3.1));
\item[(ii)] $x$ is a $(d,\tilde{X},2)$-distributionally irregular vector\index{$\tilde{X}$-distributionally irregular vector} for
$(({\mathcal A}_{j,k})_{k\in {\mathbb N}})_{1\leq j\leq N}$ iff $x$ is $d$-distributionally
near to $0$ of type $3$ for $(({\mathcal A}_{j,k})_{k\in {\mathbb N}})_{1\leq j\leq N}$ and $x$ is $d$-distributionally unbounded of type $1$ for $(({\mathcal A}_{j,k})_{k\in {\mathbb N}})_{1\leq j\leq N}$ (the use of $\cap \forall$ in (3.2));
\item[(iii)] $x$ is a $(d,\tilde{X},3)$-distributionally irregular vector\index{$\tilde{X}$-distributionally irregular vector} for
$(({\mathcal A}_{j,k})_{k\in {\mathbb N}})_{1\leq j\leq N}$ iff $x$ is $d$-distributionally
near to $0$ of type $3$ for $(({\mathcal A}_{j,k})_{k\in {\mathbb N}})_{1\leq j\leq N}$ and $x$ is $d$-distributionally unbounded of type $3$ for $(({\mathcal A}_{j,k})_{k\in {\mathbb N}})_{1\leq j\leq N}$
(the use of $\forall \forall$ in (3.3));
\item[(iv)] $x$ is a $(d,\tilde{X},4)$-distributionally irregular vector\index{$\tilde{X}$-distributionally irregular vector} for
$(({\mathcal A}_{j,k})_{k\in {\mathbb N}})_{1\leq j\leq N}$ iff $x$ is $d$-distributionally
near to $0$ of type $4$ for $(({\mathcal A}_{j,k})_{k\in {\mathbb N}})_{1\leq j\leq N}$ and $x$ is $d$-distributionally unbounded of type $3$ for $(({\mathcal A}_{j,k})_{k\in {\mathbb N}})_{1\leq j\leq N}$
(the use of $\forall \exists$ in (3.4));
\item[(v)] $x$ is a $(d,\tilde{X},5)$-distributionally irregular vector\index{$\tilde{X}$-distributionally irregular vector} for
$(({\mathcal A}_{j,k})_{k\in {\mathbb N}})_{1\leq j\leq N}$ iff $x$ is $d$-distributionally
near to $0$ of type $3$ for $(({\mathcal A}_{j,k})_{k\in {\mathbb N}})_{1\leq j\leq N}$ and $x$ is $d$-distributionally unbounded of type $4$ for $(({\mathcal A}_{j,k})_{k\in {\mathbb N}})_{1\leq j\leq N}$
(the use of $\exists \forall$ in (3.5));
\item[(vi)] $x$ is a $(d,\tilde{X},6)$-distributionally irregular vector\index{$\tilde{X}$-distributionally irregular vector} for
$(({\mathcal A}_{j,k})_{k\in {\mathbb N}})_{1\leq j\leq N}$ iff $x$ is $d$-distributionally
near to $0$ of type $4$ for $(({\mathcal A}_{j,k})_{k\in {\mathbb N}})_{1\leq j\leq N}$ and $x$ is $d$-distributionally unbounded of type $1$ for $(({\mathcal A}_{j,k})_{k\in {\mathbb N}})_{1\leq j\leq N}$
(the use of $\cap \exists$ in (3.6));
\item[(vii)] $x$ is a $(d,\tilde{X},7)$-distributionally irregular vector\index{$\tilde{X}$-distributionally irregular vector} for
$(({\mathcal A}_{j,k})_{k\in {\mathbb N}})_{1\leq j\leq N}$ iff $x$ is $d$-distributionally
near to $0$ of type $1$ for $(({\mathcal A}_{j,k})_{k\in {\mathbb N}})_{1\leq j\leq N}$ and $x$ is $d$-distributionally unbounded of type $2$ for $(({\mathcal A}_{j,k})_{k\in {\mathbb N}})_{1\leq j\leq N}$
(the use of $ \forall \cap$ in (3.7));
\item[(viii)] $x$ is a $(d,\tilde{X},8)$-distributionally irregular vector\index{$\tilde{X}$-distributionally irregular vector} for
$(({\mathcal A}_{j,k})_{k\in {\mathbb N}})_{1\leq j\leq N}$ iff $x$ is $d$-distributionally
near to $0$ of type $1$ for $(({\mathcal A}_{j,k})_{k\in {\mathbb N}})_{1\leq j\leq N}$ and $x$ is $d$-distributionally unbounded of type $4$ for $(({\mathcal A}_{j,k})_{k\in {\mathbb N}})_{1\leq j\leq N}$
(the use of $\exists \cap $ in (3.8));
\item[(ix)] $x$ is a $(d,\tilde{X},9)$-distributionally irregular vector\index{$\tilde{X}$-distributionally irregular vector} for
$(({\mathcal A}_{j,k})_{k\in {\mathbb N}})_{1\leq j\leq N}$ iff $x$ is $d$-distributionally
near to $0$ of type $1$ for $(({\mathcal A}_{j,k})_{k\in {\mathbb N}})_{1\leq j\leq N}$ and $x$ is $d$-distributionally unbounded of type $2$ for $(({\mathcal A}_{j,k})_{k\in {\mathbb N}})_{1\leq j\leq N}$
(the use of $ \cup \cap$ in (3.9));
\item[(x)] $x$ is a $(d,\tilde{X},10)$-distributionally irregular vector\index{$\tilde{X}$-distributionally irregular vector} for
$(({\mathcal A}_{j,k})_{k\in {\mathbb N}})_{1\leq j\leq N}$ iff $x$ is $d$-distributionally
near to $0$ of type $3$ for $(({\mathcal A}_{j,k})_{k\in {\mathbb N}})_{1\leq j\leq N}$ and $x$ is $d$-distributionally unbounded of type $2$ for $(({\mathcal A}_{j,k})_{k\in {\mathbb N}})_{1\leq j\leq N}$
(the use of $\cup \forall$ in (3.10));
\item[(xi)] $x$ is a $(d,\tilde{X},11)$-distributionally irregular vector\index{$\tilde{X}$-distributionally irregular vector} for
$(({\mathcal A}_{j,k})_{k\in {\mathbb N}})_{1\leq j\leq N}$ iff $x$ is $d$-distributionally
near to $0$ of type $2$ for $(({\mathcal A}_{j,k})_{k\in {\mathbb N}})_{1\leq j\leq N}$ and $x$ is $d$-distributionally unbounded of type $1$ for $(({\mathcal A}_{j,k})_{k\in {\mathbb N}})_{1\leq j\leq N}$
(the use of $\cap \cup$ in (3.11));
\item[(xii)] $x$ is a $(d,\tilde{X},12)$-distributionally irregular vector\index{$\tilde{X}$-distributionally irregular vector} for
$(({\mathcal A}_{j,k})_{k\in {\mathbb N}})_{1\leq j\leq N}$ iff $x$ is $d$-distributionally
near to $0$ of type $2$ for $(({\mathcal A}_{j,k})_{k\in {\mathbb N}})_{1\leq j\leq N}$ and $x$ is $d$-distributionally unbounded of type $3$ for $(({\mathcal A}_{j,k})_{k\in {\mathbb N}})_{1\leq j\leq N}$
(the use of $ \forall \cup$ in (3.12)).
\end{itemize}
\end{defn}

In the case that $\tilde{X}=X,$ then
we also say that $x$ is a $(d,i)$-distributionally irregular vector for
$(({\mathcal A}_{j,k})_{k\in {\mathbb N}})_{1\leq j\leq N}$ ($i\in {\mathbb N}_{12}$). We similarly define the notion of a
$(d,i)$-distributionally
near to $0$ ($(d,i)$-distributionally $m$-unbounded, $(d,i)$-distributionally unbounded, $(d,\tilde{X},i)$-distributionally irregular, $(d,i)$-distributionally irregular) vector for tuple $({\mathcal A}_{j})_{1\leq j\leq N}$ of MLOs.

We can formulate a great number of comparison principles regarding the inheritance of $(d,\tilde{X},i)$-distributional chaos ($1\leq i\leq 12$), 
$d$-distributionally
near to $0$ of type $i$ vectors and
$d$-distributionally $m$-unbounded vectors of type $i$ ($1\leq i\leq 4$) under the actions of linear topological homeomorphisms acting between the corresponding pivot spaces. Details can be left to the interested readers.

Differences between Banach spaces and Fr\' echet spaces observed in
\cite{marek-trio} also hold for disjoint distributionally unbounded vectors:

\begin{example}\label{gos}
Set $\tilde{B}:=\{k\in {\mathbb N} : {\mathcal A}_{j,k}\mbox{ is purely multivalued for all }j\in {\mathbb N}_{N}\}.$ Let $Y$ be a Banach space and let $\overline{dens}(\tilde{B})=1.$ 
Then any non-zero vector $x\in \bigcap_{j=1}^{N}\bigcap_{k=1}^{\infty}D({\mathcal A}_{j,k})$ is $(d,1)$-distributionally unbounded for
$(({\mathcal A}_{j,k})_{k\in {\mathbb N}})_{1\leq j\leq N}$. To see this, take $B=\tilde{B};$ then for any  $j\in {\mathbb N}_{N}$ and $k\in B,$ choosing arbitrary $x_{j,k}\in {\mathcal A}_{j,k}x,$  we can always find $y_{j,k} \in {\mathcal A}_{j,k}0$ such that we have
$\| x_{j,k}'\|_{Y}=\|x_{j,k}+y_{j,k}\|_{Y}>2^{k},$ with $x_{j,k}'=x_{j,k}+y_{j,k}.$
The situation is quite different in the case that $Y$ is a Fr\' echet space; towards see this, let us again assume that the set $\tilde{B}$ defined above has the upper density equal to $1.$ Then there need not exist a vector $x\in \bigcap_{j=1}^{N}\bigcap_{k=1}^{\infty}D({\mathcal A}_{j,k})$
that is distributionally $m$-unbounded for some $m\in {\mathbb N}$. To illustrate this, consider the case in which $X:=Y:=C({\mathbb R}),$ equipped with the usual topology, and the operator ${\mathcal A}_{k}$ is defined by $D({\mathcal A}_{k}):=X$
and ${\mathcal A}_{j,k}f:=f+C_{[jk,\infty)}({\mathbb R}),$ $k\in {\mathbb N}$ ($j\in {\mathbb N}_{N}$), where $C_{[jk,\infty)}({\mathbb R}) :=\{f \in C({\mathbb R}) : \mbox{supp}(f)\subseteq [jk,\infty)\}.$ Then $\tilde{B}={\mathbb N}$ but for any $f\in X$
we have $\|f+g\|_{m}^{Y}=\|f\|_{m}^{Y}\equiv \sup_{x\in [-m,m]}|f(x)|,$
$g\in C_{[jk,\infty)}({\mathbb R}),$ $m\leq jk.$
\end{example}

Furthermore, in any part of Definition \ref{DC-unbounded-fric-DISJOINT}, say the part (i), we can impose the condition
that there exist an uncountable
set $S\subseteq \bigcap_{j=1}^{N} \bigcap_{k=1}^{\infty} D({\mathcal A}_{j,k}) \cap \tilde{X}$ and
$\sigma>0$ such that for each $\epsilon>0$ and for each pair $x,\
y\in S$ of distinct points we have that for each $j\in {\mathbb N}_{N}$ and $k\in {\mathbb N}$ there exist elements $x_{j,k},\ x_{j,k}'\in {\mathcal A}_{j,k}x$ and $y_{j,k},\ y_{j,k}'\in {\mathcal A}_{j,k}y$ such that:
\begin{align}\label{bzo-kuso}
\begin{split}
& \overline{dens}\Biggl( \bigcap_{j\in {\mathbb N}} \bigl\{k \in {\mathbb N} :
d_{Y}\bigl(x_{j,k},y_{j,k}\bigr)\geq \sigma \bigr\}\Biggr)=1,
\\
& \overline{dens}\Biggl( \bigcap_{j\in {\mathbb N}} \bigl\{k \in {\mathbb N} : d_{Y}\bigl(x_{j,k}',y_{j,k}'\bigr)
<\epsilon \bigr\}\Biggr)=1;
\end{split}
\end{align}   
for these purposes, the adaptation of notion introduced in the last three definitions is obvious (see also the part (i) of Definition \ref{iregularni-prc}). If this is the case, i.e., if we accept the validity of \eqref{bzo-kuso} in place of \eqref{d-jednacina}, let us say that the sequence $({\mathcal A}_{k})_{k\in
{\mathbb N}}$ is weakly
$(\tilde{X},1)$-distributionally chaotic, etc.. It should be noted that the weak $(\tilde{X},1)$-distributional chaos is a very mild sort of chaos for MLOs. Speaking-matter-of-factly, one can employ the analysis from Example \ref{gos} for proving the existence of a substantially large
class of densely weak $(d,1)$-distributionally chaotic MLOs in Banach spaces (for
$d$-chaoticity and $d$-topological mixing property of these operators, one has to assume some extra conditions in the equation \eqref{svij} below,
like the intersection of sets $P_{j}/Q_{j}( \Lambda' )$ with the unit circle and its exterior; cf. \cite{kerry-drew} for more details). For example, we have the following:

\begin{example}\label{deds-MLOS-vg}
Let $X$ be a Banach space.
Suppose that $A$ is
a closed linear operator on $X$ satisfying that there exist an
open connected subset $\emptyset \neq \Lambda$ of ${\mathbb K}={\mathbb C}$
and an analytic mapping $g : \Lambda \rightarrow X \setminus
\{0\}$ such that $A g(\lambda)=\lambda g(\lambda),$ $\lambda \in
\Lambda .$ Let $P_{j}(z)$ and $Q_{j}(z)$ be non-zero complex
polynomials ($1\leq j\leq N$), let $R:=\{z\in {\mathbb C} :
Q_{j}(z)=0\mbox{ for some }j\in {\mathbb N}_{N}\},$ $\Lambda':=\Lambda \setminus R ,$
and let $X':=\text{span}\{  g(\lambda) : \lambda \in \Lambda' \},$ $\tilde{X}:=\overline{X'}.$
Suppose that
\begin{align}\label{svij}
\frac{P_{j}}{Q_{j}}\bigl( \Lambda' \bigr) \subseteq \{ z\in {\mathbb C} : |z|<1 \},\quad j\in {\mathbb N}_{N}
\end{align}
and that, for every $j\in {\mathbb N}_{N},$ the operators
$Q_{j}(A)^{-1}P_{j}(A)$ and $P_{j}( A)Q_{j}(A)^{-1}$ are purely multivalued. Let ${\mathcal A}_{j}$ be either
$Q_{j}(A)^{-1}P_{j}(A)$ or $P_{j}( A)Q_{j}(A)^{-1}$ ($1\leq j\leq N$).
Since for any MLO ${\mathcal A}$ we have ${\mathcal A}0\subseteq {\mathcal A}^{k}0,$ $k\in {\mathbb N},$ the operator
${\mathcal A}_{j}^{k}$ is also purely multivalued ($1\leq j\leq N,$ $k\in {\mathbb N}$). Further on,
for each $\lambda \in \Lambda'$ we can inductively prove that
$$
\Biggl[\frac{P_{j}}{Q_{j}}( \lambda)\Biggr]^{k} g(\lambda) \in  {\mathcal A}_{j}^{k}g(\lambda),\quad k\in {\mathbb N}.
$$
This simply implies on account of \eqref{svij} and the analysis from Example \ref{gos} that any vector
$x\in X'$ is weakly $(d,\tilde{X},1)$-distributionally irregular for the operators
${\mathcal A}_{1},\cdot \cdot \cdot, {\mathcal A}_{N}$ as well as that 
 the operators ${\mathcal A}_{1},\cdot \cdot \cdot, {\mathcal A}_{N}$
are densely weak $(\tilde{X},1)$-distributionally chaotic.
\end{example}

On the other hand, there exists a substantially large class of purely multivalued linear operators that are not weakly $(d,i)$-distributionally chaotic for any $i\in {\mathbb N}_{12}:$

\begin{example}\label{qwer}
(cf. also \cite[Example 4.6]{kerry-drew}) 
Suppose that $A_{1}\in L(X),\cdot \cdot \cdot , A_{N}\in L(X).$
Then any MLO extension of tuple $(A_{1},\cdot \cdot \cdot , A_{N}),$ i.e, any tuple $({\mathcal A}_{1},\cdot \cdot \cdot, {\mathcal A}_{N})$ of MLOs such that $A_{j}\subseteq {\mathcal A}_{j}$ for all $j\in {\mathbb N}_{N},$ has the form $(A_{1}+W_{1},\cdot \cdot \cdot , A_{N}+W_{N}),$ where $W_{i}$ is a linear submanifold of $X$ ($ 1\leq i\leq N$). We already know that an extension
$({\mathcal A}_{1},\cdot \cdot \cdot , {\mathcal A}_{N})
$ of $(A_{1},\cdot \cdot \cdot , A_{N})$ is $d$-topologically mixing if any $W_{j}$ is dense in $X$ ($1\leq j\leq N$). Let us recall that
for any MLO extension ${\mathcal A}=A+W$ of operator $A\in L(X)$ we have
\begin{align*}
{\mathcal A}^{n}x=A^{n}x+\sum \limits_{j=0}^{n-1}A^{j}(W),\quad n\in {\mathbb N},\ x\in X.
\end{align*}
Suppose, further, that $X$ is an infinite-dimensional complex
Hilbert space with the complete orthonormal basis $\{e_{n} : n\in {\mathbb N}\}.$ Put
$A\sum_{n=1}^{\infty}x_{n}e_{n}:=\sum_{n=1}^{\infty}x_{n}e_{n+1},$ for any $x=\sum_{n=1}^{\infty}x_{n}e_{n}\in X.$
Then $(A+W_{1},A^{2}+W_{2},\cdot \cdot \cdot , A^{N}+W_{N})$ is not a weak $(d,i)$-distributionally chaotic extension of
$(A,A^{2},\cdot \cdot \cdot , A^{N})$, where $W_{j}$ is the linear span of
$\{e_{1},e_{2},\cdot \cdot \cdot,e_{j}\}$ ($i\in {\mathbb N}_{12}$, $ 1\leq j\leq N$). 
To see this, it is only worth observing that ($j\in {\mathbb N}_{N},$ $k\in {\mathbb N}$):
$$
\bigl(A^{j}+W_{j} \bigr)^{k}\sum_{n=1}^{\infty}x_{n}e_{n}=\sum_{n=1}^{\infty}x_{n}e_{n+jk}+\mbox{span}\bigl\{e_{1},\cdot \cdot \cdot,e_{jk}\bigr\},\quad x=\sum_{n=1}^{\infty}x_{n}e_{n}\in X,
$$
$\|z\|\geq |x_{n}|$ for all $n\in {\mathbb N}$ and $z\in (A^{j}+W_{j})^{k}\sum_{n=1}^{\infty}x_{n}e_{n}.$
\end{example}

It is also clear that there exists a great number of purely multivalued linear operators that are weakly $(d,i)$-distributionally chaotic but not $(d,i)$-distributionally chaotic ($i\in {\mathbb N}_{12}$). For example, let $i=1,$ 
$X:={\mathbb K}^{n},$ $W$ be a non-trivial subspace of $X$ and 
${\mathcal A}_{j}:=I+W$ for all $j\in {\mathbb N}_{N}.$ Then it can be easily seen that the tuple $({\mathcal A}_{1},\cdot \cdot \cdot, {\mathcal A}_{j})$ is weakly $(d,1)$-distributionally chaotic (with the scrambled set $S=W$) but not $(d,1)$-distributionally chaotic.
Primarily from the practical point of view, Definition \ref{DC-unbounded-fric-DISJOINT} and Definition \ref{iregularni-prc} will be our general framework for further considerations of disjoint distributional chaos. 

Let $\{0\} \neq X' \subseteq \tilde{X}$ be a linear manifold, and let $i\in {\mathbb N}_{12}.$ Then
we say that:
\begin{itemize}
\item[d1.]
 $X'$ is a $(d,\tilde{X},i)$-distributionally
irregular manifold\index{$(d,\tilde{X},i)$-distributionally
irregular manifold} for $(({\mathcal A}_{j,k})_{k\in {\mathbb N}})_{1\leq j\leq N}$
($(d,i)$-distributionally irregular manifold\index{distributionally irregular manifold} in the case that $\tilde{X}=X$)
iff any element $x\in (X' \cap
\bigcap_{j=1}^{N}\bigcap_{k=1}^{\infty}D({\mathcal A}_{j,k})) \setminus \{0\}$ is a
$(d,\tilde{X},i)$-distributionally irregular vector for
$(({\mathcal A}_{j,k})_{k\in {\mathbb N}})_{1\leq j\leq N};$ the notion of a ($(d,i)$-, $(d,\tilde{X},i)$-)distributionally irregular manifold for $({\mathcal A}_{j})_{1\leq j\leq N}$ is defined similarly.
\item[d2.] $X'$ is a uniformly $(d,\tilde{X},i)$-distributionally
irregular manifold\index{$\tilde{X}$-distributionally
irregular manifold!uniformly} for\\ $(({\mathcal A}_{j,k})_{k\in {\mathbb N}})_{1\leq j\leq N}$
(uniformly $(d,i)$-distributionally irregular manifold\index{distributionally irregular manifold!uniformly} in the case that $\tilde{X}=X$)
iff there exists $m\in {\mathbb N}$ such that
any vector $x\in (X' \cap
\bigcap_{j=1}^{N}
\bigcap_{k=1}^{\infty}D({\mathcal A}_{j,k})) \setminus \{0\}$ is both $(d,i)$-distributionally $m$-unbounded (with the meaning clear) and $(d,i)$-distributionally near to $0$ for $(({\mathcal A}_{j,k})_{k\in {\mathbb N}})_{1\leq j\leq N}.$
In this case, $X'$ is $2^{-m}_{\tilde{X}}$-scrambled set for
$(({\mathcal A}_{j,k})_{k\in {\mathbb N}})_{1\leq j\leq N}.$
\end{itemize}
As expected, the existence of disjoint distributionally irregular vectors implies the existence of uniformly disjoint distributionally irregular manifolds. More precisely, we have the following:
\begin{itemize}
\item[d3.]
Suppose that $0\neq x\in \tilde{X} \cap \bigcap_{j=1}^{N}\bigcap_{k=1}^{\infty}D({\mathcal A}_{j,k})$ is a
$(d,\tilde{X},i)$-distributionally irregular vector for $(({\mathcal A}_{j,k})_{k\in {\mathbb N}})_{1\leq j\leq N}.$
Then $X'\equiv span\{x\}$
is a uniformly $(d,\tilde{X},i)$-distributionally irregular manifold for
$(({\mathcal A}_{j,k})_{k\in {\mathbb N}})_{1\leq j\leq N}.$
\end{itemize}
If $X'$ is dense in $\tilde{X},$
then the notions of dense ($(d,i)$-, $(d,\tilde{X},i)$-)distributionally
irregular manifolds, dense uniformly ($(d,i)$-, $(d,\tilde{X},i)$-)distributionally
irregular manifolds, etc., are defined analogically.
It will be said that $({\mathcal A}_{1,k})_{k\in {\mathbb N}}, ({\mathcal A}_{2,k})_{k\in {\mathbb N}},\cdot \cdot \cdot, ({\mathcal A}_{N,k})_{k\in {\mathbb N}}$ are
$(d,\tilde{X},i)$-distributionally chaotic
iff the tuple $(({\mathcal A}_{j,k})_{k\in {\mathbb N}})_{1\leq j\leq N}$ is
$(d,\tilde{X},i)$-distributionally chaotic; a similar terminological agreement will be accepted for operators.

\begin{rem}\label{prc-qwe}
\begin{itemize}
\item[(i)]
If $i\in \{1,2,3,7\},$ resp. $i\in\{4,5,6,8\}$, and $(({\mathcal A}_{j,k})_{k\in
{\mathbb N}})_{1\leq j\leq N}$ is (densely) $(d,\tilde{X},i)$-distributionally chaotic, then for each $j\in {\mathbb N}_{N},$ resp. there exists $j\in {\mathbb N}_{N},$ such that the component $({\mathcal A}_{j,k})_{k\in
{\mathbb N}}$ is (densely) $(\tilde{X},i)$-distributionally chaotic. Furthermore, if we assume that $X'$ is a (uniformly) $(d,\tilde{X},i)$-distributionally
irregular manifold for $(({\mathcal A}_{j,k})_{k\in {\mathbb N}})_{1\leq j\leq N},$ then for each $j\in {\mathbb N}_{N},$ resp. there exists $j\in {\mathbb N}_{N},$ we have that $X'$ is a (uniformly) $\tilde{X}$-distributionally
irregular manifold for $({\mathcal A}_{j,k})_{k\in
{\mathbb N}}.$ Similar statements hold for $(d,i)$-distributionally
near to $0$ vectors, $(d,i)$-distributionally ($m$-)unbounded vectors and $(d,\tilde{X},i)$-distributionally irregular vectors.
\item[(ii)] Let $i\in {\mathbb N}_{12}.$ It is well known that a single-valued linear operator $A$ and its constant multiple $cA$ cannot be $d$-hypercyclic (cf. \cite[p. 299]{bp07} for continuous case, and \cite{kerry-drew} for general case). This is no longer true for $(d,i)$-distributional chaos because
our notion allows that a (densely) $(d,i)$-distributionally chaotic sequence $(({\mathcal A}_{j,k})_{k\in
{\mathbb N}})_{1\leq j\leq N}$ can have the same components (any of them can be repeated a certain finite number of times). Speaking-matter-of-factly, we are interested in finding $(d,i)$-distributionally chaotic tuples whose components are strictly different and
which are not rotations of some other components (as it is well known, distributional chaos is invariant under rotations).
\end{itemize}
\end{rem}

It is not difficult to verify that the notions of $(d,i_{1})$-distributional chaos and $(d,i_{2})$-distributional chaos differ for the sequences of 
continuous linear operators,
provided that $i_{1},\ i_{2}\in {\mathbb N}_{12}$ and $i_{1}\neq i_{2}$ (the same statement holds for $(d,i_{1})$-distributionally irregular vectors and $(d,i_{2})$-distributionally irregular vectors, as well). Because of completeness of our study, we have decided to thoroughly 
illustrate this fact by a series of plain and elaborate examples: 

\begin{example}\label{2}
Suppose that $i=2.$ Then $(d,2)$-distributional chaos implies $(d,i)$-distributional chaos for $i\in \{3,4,5,6,10,11,12\}$ and we need to prove that 
$(d,2)$-distributional chaos does not imply $(d,i)$-distributional chaos for $i\in \{1,7,8,9\}.$ Towards this end,
set $X:={\mathbb K}^{n}.$ Take any 
two disjoint subsets $A$ and $B$ of ${\mathbb N}$ such that ${\mathbb N}=A\cup B$ and $\overline{dens}(A)=\overline{dens}(B)=1.$ If 
$k\in A$ and $j\in {\mathbb N}_{N},$ set
$T_{j,k}=0.$ Further on, it is clear that there exist pairwise disjoint subsets $B_{1},\cdot \cdot \cdot,B_{N}$ of $B$ such that 
$\cup_{j=1}^{N}B_{j}=B$
and $\overline{dens}(B_{j})=1$ for all $j\in {\mathbb N}_{N}.$ For any $k\in B,$ set 
$T_{j,k}:=0$ if $k\notin B_{j}$ and 
$T_{j,k}:=diag(a_{11}^{jk},\cdot \cdot \cdot,a_{nn}^{jk})$ such that $|a_{ii}^{jk}|\geq j+k$ for all $i\in {\mathbb N}_{n},$ otherwise ($j\in {\mathbb N}_{N}$). Then (3.2) holds with $S=X$ and $\sigma=1$ while (3.1), (3.7), (3.8) and (3.9) do not hold.
\end{example}

\begin{example}\label{3}
Suppose that $i=3.$ Then $(d,3)$-distributional chaos implies $(d,i)$-distributional chaos for $i\in\{4,5,10,12\}$ and we need to prove that 
$(d,3)$-distributional chaos does not imply $(d,i)$-distributional chaos for any integer $i\in \{1,2,6,7,8,9,11\}.$ If $i\in \{7,8,9\},$ then we can argue as in the previous example; if $i\in \{1,2,6,11\},$
then we can argue as in Example \ref{7} below.
\end{example}

\begin{example}\label{4}
Suppose that $i=4.$ In this case, the only consequence of $(d,4)$-distributional chaos is $(d,12)$-distributional chaos. To see that $(d,4)$-distributional chaos does not imply $(d,i)$-distributional chaos for $i\in \{1,2,3,5,7,8,9,10\},$ we can take any $T_{1}\in L(X)$ that is distributionally chaotic and put $T_{j}:=2I$ for $j\in {\mathbb N}_{N}\setminus \{1\}.$ To see that $(d,4)$-distributional chaos does not imply $(d,i)$-distributional chaos for $i\in \{6,11\},$ set $X:={\mathbb K}^{n}.$ After that, take any 
two disjoint subsets $A$ and $B$ of ${\mathbb N}$ such that ${\mathbb N}=A\cup B$ and $\overline{dens}(A)=\overline{dens}(B)=1.$ If $k\in B$ and $j\in {\mathbb N}_{N},$ set $T_{j,k}:=0.$ Further on, it is clear that there exist pairwise disjoint subsets $A_{1},\cdot \cdot \cdot,A_{N}$ of $A$ such that 
$\cup_{j=1}^{N}A_{j}=A$
and $\overline{dens}(A_{j})=1$ for all $j\in {\mathbb N}_{N}.$ If $k\in A_{j}$ for some $j\in {\mathbb N}_{N},$ we define $T_{j,k}:=diag(a_{11}^{jk},\cdot \cdot \cdot,a_{nn}^{jk})$ such that $|a_{ii}^{jk}|\geq j+k$ for all $i\in {\mathbb N}_{n};$ if $j\in {\mathbb N}_{N}$ and $k\notin A_{j},$ we define 
$T_{j,k}:=0$. Then (3.4) holds with $S=X$ and $\sigma=1$ while (3.6) and (3.11) do not hold.
\end{example}

\begin{example}\label{5}
Suppose that $i=5.$ In this case, the only consequence of $(d,5)$-distributional chaos is $(d,10)$-distributional chaos. To see that $(d,5)$-distributional chaos does not imply $(d,i)$-distributional chaos for $i\in \{1,2,3,4,6,7,11,12\},$ we can take any $T_{1}\in L(X)$ that is distributionally chaotic and put $T_{j}:=0$ for $j\in {\mathbb N}_{N}\setminus \{1\}.$ To see that $(d,5)$-distributional chaos does not imply $(d,i)$-distributional chaos for $i\in \{8,9\},$ we can simply set $X:={\mathbb K}^{n}$
and use the procedure similar to those ones employed in Example \ref{2} and Example \ref{4}.
%
%
\end{example}

\begin{example}\label{6}
Suppose that $i=6.$ In this case, the consequence of $(d,6)$-distributional chaos is $(d,i)$-distributional chaos for any $i\in \{4,11,12\}$. To see that $(d,6)$-distributional chaos does not imply $(d,i)$-distributional chaos for $i\in \{1,2,3,5,7,8,9,10\}$ we can take, as in the case that $i=4,$ any $T_{1}\in L(X)$ that is distributionally chaotic and put $T_{j}:=2I$ for $j\in {\mathbb N}_{N}\setminus \{1\}.$ 
\end{example}

\begin{example}\label{7}
Suppose that $i=7.$ All that we need to show is that $(d,7)$-distributional chaos does not imply $(d,i)$-distributional chaos for any $i\in \{1,2,6,11\}$. To see this,
we can set $X:={\mathbb K}^{n}$
and slightly modify Example \ref{2} and Example \ref{4}.
%
%
\end{example}

\begin{example}\label{8}
Suppose that $i=8.$ In this case, the consequence of $(d,8)$-distributional chaos is 
$(d,i)$-distributional chaos for $i\in\{5,9,10\}$. In order to see that $(d,8)$-distributional chaos does not imply $(d,i)$-distributional chaos for $i\in \{1,2,3,4,6,7,11,12\}$ we can take, as in the case that $i=5,$ any $T_{1}\in L(X)$ that is distributionally chaotic and put $T_{j}:=0$ for $j\in {\mathbb N}_{N}\setminus \{1\}.$ 
\end{example}

\begin{example}\label{9}
Suppose that $i=9.$ In this case, the  only  consequence of $(d,9)$-distributional chaos is 
$(d,10)$-distributional chaos. To see that $(d,9)$-distributional chaos does not imply $(d,i)$-distributional chaos for $i\in \{1,2,3,4,6,7,11,12\}$ we can reexamine the first example from the case 
$i=8.$ 
To see that $(d,9)$-distributional chaos does not imply $(d,i)$-distributional chaos for $i\in \{5,8\},$ set $X:={\mathbb K}^{n}.$ After that, take any 
two disjoint subsets $A$ and $B$ of ${\mathbb N}$ such that ${\mathbb N}=A\cup B$ and $\overline{dens}(A)=\overline{dens}(B)=1.$ If $k\in B$ and $j\in {\mathbb N}_{N},$ set $T_{j,k}:=0.$ Further on, it is clear that there exist pairwise disjoint subsets $A_{1},\cdot \cdot \cdot,A_{N}$ of $A$ such that 
$\cup_{j=1}^{N}A_{j}=A$
and $\overline{dens}(A_{j})<1$ for all $j\in {\mathbb N}_{N}.$ For any $k\in A,$ set 
$T_{j,k}:=0$ if $k\notin A_{j}$ and 
$T_{j,k}:=diag(a_{11}^{jk},\cdot \cdot \cdot,a_{nn}^{jk})$ such that $|a_{ii}^{jk}|\geq j+k$ for all $i\in {\mathbb N}_{n},$ otherwise. Then (3.9) holds with $S=X$ and $\sigma=1$ while (3.5) and (3.8) do not hold.
\end{example}

\begin{example}\label{10}
Suppose that $i=10.$ Since $(d,5)$-distributional chaos implies $(d,10)$-distributional chaos, our analysis from Example \ref{5} shows that $(d,10)$-distributional chaos does not imply $(d,i)$-distributional chaos
for any integer $i\in \{1,2,3,4,6,7,8,9,11,12\}.$ All that remains to be shown is that $(d,10)$-distributional chaos does not imply $(d,5)$-distributional chaos. Set $X:={\mathbb K}^{n}$ and slightly modify the procedure from Example \ref{9}.
%
\end{example}

\begin{example}\label{11}
Suppose that $i=11.$ Then $(d,11)$-distributional chaos implies $(d,12)$-distributional chaos and using the first example presented in Example \ref{4} we get that $(d,11)$-distributional chaos does not imply $(d,i)$-distributional chaos
for any $i\in \{1,2,3,5,7,8,9,10\}.$ Now we will prove that $(d,11)$-distributional chaos does not imply $(d,4)$-distributional chaos and $(d,6)$-distributional chaos. We can simply set $X:={\mathbb K}^{n}$ and slightly modify Example \ref{9}.
\end{example}

\begin{example}\label{12}
Suppose that $i=12.$ In this case, $(d,4)$-distributional chaos implies $(d,12)$-distributional chaos so that $(d,12)$-distributional chaos does not imply $(d,i)$-distributional chaos for $i\in \{1,2,3,5,7,8,9,10\}.$
To see that $(d,12)$-distributional chaos does not imply $(d,i)$-distributional chaos for $i\in \{4,6\},$ we can argue as in the previous example; to see that $(d,12)$-distributional chaos does not imply $(d,11)$-distributional chaos, we can argue as in Example \ref{7}.
\end{example}

In Example \ref{2}-Example \ref{12}, we have used the sequences of continuous linear operators. It is expected that, for continuous linear operators $T_{1}\in L(X),\cdot \cdot \cdot,T_{N}\in L(X),$ the notions of $(d,i_{1})$-distributional chaos and $(d,i_{2})$-distributional chaos ($(d,i_{1})$-distributionally irregular vectors and $(d,i_{2})$-distributionally irregular vectors)
do not coincide for different values of indexes  $i_{1},\ i_{2}\in {\mathbb N}_{12},$ as well. We will present only two illustrative examples concerning this question:

\begin{example}\label{sunce} (cf. \cite[Remark 21]{2011})
Consider
a weighted forward shift $F_{\omega}: 
l^{2} \rightarrow l^{2},$ defined by $F_{\omega} (x_{1}, x_{2}, \cdot \cdot \cdot) \mapsto (0, \omega_{1}x_{1},\omega_{2}x_{2},\cdot \cdot \cdot),$ where the sequence of weights $\omega =(\omega_{k})_{k\in {\mathbb N}}$ consists of sufficiently large blocks of $2$'s and blocks of $(1/2)$'s such that the vector $e_{1} = (1, 0, \cdot \cdot \cdot)$ is a
distributionally irregular vector for $F_{\omega}$.
To precise this, assume that $(a_{n})_{n\in {\mathbb N}}$ and $(b_{n})_{n\in {\mathbb N}}$ are two sequences of natural numbers such that:
\begin{itemize}
\item[(i)] $1<b_{1}<a_{1}<b_{2}<a_{2}<\cdot \cdot \cdot ;$
\item[(ii)] there exists $n_{0}\in {\mathbb N}$ such that $b_{n}>n\sum_{i=1}^{n-1}(a_{i}-b_{i})+n^{2}+n+1$ and $a_{n}>n\sum_{i=1}^{n-1}(a_{i}-b_{i})+nb_{n}+n^{2}+n+1$ for all $n\in {\mathbb N}$ with $n\geq n_{0};$
\item[(iii)] $\limsup_{n\rightarrow \infty}\frac{\sum_{i=1}^{n}a_{i}(1-i^{-1})}{\sum_{i=1}^{n}(a_{i}+b_{i})}=1$ and $\limsup_{n\rightarrow \infty}\frac{\sum_{i=1}^{n}b_{i}(1-i^{-1})}{b_{n}+\sum_{i=1}^{n-1}(a_{i}+b_{i})}=1.$ 
\end{itemize}
Define the sequence of weights $\omega =(\omega_{k})_{k\in {\mathbb N}}$ in the following way: $\omega_{k}=:2$ iff $k\in [1,b_{1}]$ or there exists an integer $l\in {\mathbb N}$ such that $k\in [1+\sum_{i=1}^{l}(a_{i}+b_{i}),b_{l+1}+\sum_{i=1}^{l}(a_{i}+b_{i})];$ otherwise, we set
$\omega_{k}:=1/2.$ Then the sets 
$$
A:={\mathbb N} \cap \bigcup_{n\in {\mathbb N}}\Biggl[ 1+\sum_{i=1}^{n-1}(a_{i}+b_{i})+b_{n}+n^{-1}a_{n}+n,\sum_{i=1}^{n}(a_{i}+b_{i}) \Biggr]:={\mathbb N} \cap \bigcup_{n\in {\mathbb N}}\bigl[ A_{n}^{1},A_{n}^{2} \bigr]
$$
and
$$
B:={\mathbb N} \cap \bigcup_{n\in {\mathbb N}}\Biggl[ 1+\sum_{i=1}^{n-1}(a_{i}+b_{i})+n^{-1}b_{n}+n,b_{n}+\sum_{i=1}^{n-1}(a_{i}+b_{i}) \Biggr]:={\mathbb N} \cap \bigcup_{n\in {\mathbb N}}\bigl[ B_{n}^{1},B_{n}^{2} \bigr]
$$
have the upper densities equal to $1$ because of condition (iii); here we only want to note that the correponding sequence in the definition of upper denisty of $A$ ($B$) can be chosen to be $n_{k}=\sum_{i=1}^{n_{k}}(a_{i}+b_{i})$. 
In order to see that $\lim_{n\in A}F_{\omega}e_{1}=0$ ($\lim_{n\in B}\|F_{\omega}e_{1}\|=\infty$), it suffices to observe that for each $n\in {\mathbb N}$ with $n\geq n_{0}$ and $n\in [ A_{n}^{1},A_{n}^{2}]$ ($n\in [ B_{n}^{1},B_{n}^{2}]$) we have $\omega_{1}\omega_{2}\cdot \cdot \cdot \omega_{n}<2^{-n}$ ($\omega_{1}\omega_{2}\cdot \cdot \cdot \omega_{n}>2^{n}$). A concrete example can be simply given following the analysis contained in \cite[Example 10]{gimenez-p}; we can take $b_{1}=2^{1^{2}},$ $a_{1}=2^{1^{2}}+2^{2^{2}},$ $b_{2}=2^{1^{2}}+2^{2^{2}}+2^{3^{2}},$ $a_{2}=2^{1^{2}}+2^{2^{2}}+2^{3^{2}}+2^{4^{2}},\cdot \cdot \cdot .$ 

Define now $\sigma :=(1/\omega_{k})_{k\in {\mathbb N}}.$ Then $e_{1} = (1, 0, \cdot \cdot \cdot)$ is a
distributionally irregular vector for $F_{\sigma},$ as well,
because $\lim_{n\in B}F_{\omega}e_{1}=0$ and $\lim_{n\in A}\|F_{\omega}e_{1}\|=\infty$. On account of this, the operators $F_{\omega}$ and $F_{\sigma}$ are $(d,i)$-distributionally chaotic for any $i\in \{3,4,5,10,12\}$ (with the scrambled set $S=span\{e_{1}\}$).
On the other hand, for any integer $i\in \{1,7,8,9\}$ we have that the operators $F_{\omega}$ and $F_{\sigma}$ are not $(d,i)$-distributionally chaotic. To see this, it suffices to observe that the second equality in (3.1), for such values of index $i,$ is violated. Strictly speaking, for any non-zero vector $\langle x_{n}\rangle_{n\in {\mathbb N}}\in l^{2}$ there exists $n_{0}\in {\mathbb N}$ such that $x_{n_{0}}\neq 0$ and, for every integer $k\in {\mathbb N},$ we have 
\begin{align*}
\frac{1}{|x_{n_{0}}|}\Bigl \|F_{\omega}^{k}\langle x_{n}\rangle_{n\in {\mathbb N}}+F_{\sigma}^{k}\langle x_{n}\rangle_{n\in {\mathbb N}}\Bigr\|&\geq \omega_{n_{0}}\cdot \cdot \cdot \omega_{k+n_{0}}+\sigma_{n_{0}}\cdot \cdot \cdot \sigma_{k+n_{0}}
\\ & \geq \min_{1\leq j\leq k}\Bigl(2^{2j-k}+2^{k-2j}\Bigr)\geq 2.
\end{align*}
Finally, it is worth noting that for each $k\in {\mathbb N}$ we have $\|F_{\omega}^{k}\| =\sup_{n\in {\mathbb N}}\omega_{n}\omega_{n+1}\cdot \cdot \cdot \omega_{n+k} \geq 2^{k}$ and 
$\|F_{\sigma}^{k}\| =\sup_{n\in {\mathbb N}}\sigma_{n}\sigma_{n+1}\cdot \cdot \cdot \sigma_{n+k}\geq 2^{k},$ so that Proposition \ref{owq-prcko} below implies the existence of a vector $x\in l^{2}$ such that $\lim_{k\rightarrow \infty}\|F_{\omega}^{k}x\|=\lim_{k\rightarrow \infty}\|F_{\sigma}^{k}x\|=+\infty$ (the question whether the operators $F_{\omega}$ and $F_{\sigma}$ are $(d,i)$-distributionally chaotic for $i\in \{2,6,11\}$ is interested, but we will not analyze it here). 
\end{example}  

\begin{example}\label{count-grof}
In \cite[Theorem 3.7]{countable}, Z. Yin, S. He and Y. Huang have shown that, for any two positive real numbers $a$ and $b$ such that $a<b,$ there exists an invertible operator $T$ acting on a Hilbert space $X$ such that $[a, b] =\{\lambda >0 : \lambda T \mbox{ is distributionally chaotic}\}$ and for any distinct values $\lambda_{1},\ \lambda_{2} \in [a, b],$ the operators $\lambda_{1}T$ and $\lambda_{2}T$ have no common Li-Yorke irregular vectors (see e.g. \cite[Definition 3]{countable} for the notion). Let $\lambda_{1}<\lambda_{2}$ and $\lambda_{1},\ \lambda_{2} \in [a, b].$ Then it is immediate from definition that $\lambda_{1}T$ and $\lambda_{2}T$ are $(d,i)$-distributionally chaotic for $i\in \{5,8,9,10\}$ iff $\lambda_{2}T$ is distributionally chaotic, which is the case, as well as that $\lambda_{1}T$ and $\lambda_{2}T$ are $(d,i)$-distributionally chaotic for $i\in \{4,6,11,12\}$ iff $\lambda_{1}T$ is distributionally chaotic, which is the case.
On the other hand, these operators cannot be $(d,1)$-distributionally chaotic because, if we suppose the contrary, then any non-zero vector $x\in S-S,$ where $S$ denotes the corresponding $(d,1)$-scrambled set for the operators  
$\lambda_{1}T$ and $\lambda_{2}T,$ will be a distributionally irregular vector of the operator $\lambda T,$ for any $\lambda \in (\lambda_{1},\lambda_{2}).$ This clearly contradicts the above-mentioned theorem. Furthermore,
the operators $\lambda_{1}T$ and $\lambda_{2}T$ cannot be $(d,7)$-distributionally chaotic because, if we suppose the contrary, any non-zero vector $x\in S-S,$ where $S$ denotes the corresponding $(d,7)$-scrambled set for the operators  
$\lambda_{1}T$ and $\lambda_{2}T,$ will be a Li-Yorke irregular vector for both operators $\lambda_{1} T$ and $\lambda_{2}T,$ which again contradicts the above-mentioned theorem. Finally, let us show that $\lambda_{1}T$ and $\lambda_{2}T$ cannot be $(d,2)$-distributionally chaotic or $(d,3)$-distributionally chaotic. If we suppose the contrary, then
for each non-zero vector $z\in S-S
$ there exist two strictly increasing sequences $(n_{k})$ and $(l_{k})$
of positive integers such that $\lim_{k\rightarrow \infty}\|(\lambda_{j}T)^{n_{k}}z\|=0$ and $\limsup_{k\rightarrow \infty}\|(\lambda_{j}T)^{l_{k}}z\|>0$ ($j=1,2$); here, $S$ denotes the corresponding $(d,\sigma,i)$-scrambled set. By the proofs of \cite[Theorem 3.3, Theorem 3.7]{countable}, this would imply that there exists a constant $c(\lambda_{1},\lambda_{2}),$ independent of $z,$ such that $\|(\lambda_{1}T)^{n}z\|\leq c(\lambda_{1},\lambda_{2})\|z\|$ for all $n\in {\mathbb N}$ and therefore
$\|z\|\geq \sigma /c(\lambda_{1},\lambda_{2}).$ This is a contradiction because the set $S-S$ cannot be bounded away from zero.
\end{example}

Concerning the last example, we would like to note that Z. Yin and Y. Huang have recently proved in \cite{studia-china} that for any open set $U\subseteq (0,\infty)$ which is bounded away from zero, there exists a bounded linear operator $T$ on $l^{p},$ where $1\leq p<\infty,$ such that $U=\{\lambda >0 : \lambda T \mbox{ is distributionally chaotic}\}.$ 
Motivated by the results achieved in \cite{studia-china}, for any integer $i\in {\mathbb N}_{12},$ any tuple $\vec{r}=(r_{1},r_{2},\cdot \cdot \cdot,r_{N})\in {\mathbb N}^{N}$ and any multivalued linear operator ${\mathcal A}$ on a Fr\' echet space $X,$ we introduce the set 
\begin{align*}
DDC_{{\mathcal A},i,\vec{r}}:=\Bigl\{ \vec{\lambda}=\bigl(\lambda_{1},\lambda_{2},\cdot \cdot \cdot,\lambda_{N}\bigr) \in {\mathbb K}^{N} : & \mbox{ the tuple }\bigl( \lambda_{1}{\mathcal A}^{r_{1}},\lambda_{2}{\mathcal A}^{r_{2}},\cdot \cdot \cdot,\lambda_{N}{\mathcal A}^{r_{N}} \bigr)
\\ & \mbox{ is }(d,i)\mbox{-distributionally  chaotic} \Bigr\}.
\end{align*}
Describing the structure of set $
DDC_{{\mathcal A},i,\vec{r}}$ is quite non-trivial and requires a series of further analyses.
The existence of invariant $(d,i)$-distributionally
scrambled sets and the existence of common $(d,i)$-distributionally irregular vectors for tuples of multivalued linear operators are delicate problems that will not be discussed here, as well (see \cite{afa-raj}-\cite{studia-china} and references quoted therein for further information in this direction).

We close this section by stating the following simple proposition, which has been already considered in \cite{marek-trio} for the case that $N=1;$ 
for $(({\mathcal A}_{j,k})_{k\in {\mathbb N}})_{1\leq j\leq N}$ given in advance, we
define $(({\mathbb A}_{j,k})_{k\in {\mathbb N}})_{1\leq j\leq N}$ by ${\mathbb A}_{j,k}:=({\mathcal A}_{j,k})_{|\tilde{X}}$ ($k\in {\mathbb N},\ 1\leq j\leq N$):

\begin{prop}\label{subspace-fric-DISJOINT}
Let $i\in {\mathbb N}_{12},$ let $\tilde{X}$ be a closed linear subspace of $X,$ and let $\{0\} \neq X'$ be a linear subspace of $\tilde{X}$.
\begin{itemize}
\item[(i)] The sequence $(({\mathcal A}_{j,k})_{k\in {\mathbb N}})_{1\leq j\leq N}$ is $(d,\tilde{X},i)$-distributionally chaotic
iff the sequence $(({\mathbb A}_{j,k})_{k\in {\mathbb N}})_{1\leq j\leq N}$ is $(d,i)$-distributionally chaotic.
\item[(ii)] A vector $x$ is a $(d,\tilde{X},i)$-distributionally irregular
vector for $(({\mathcal A}_{j,k})_{k\in {\mathbb N}})_{1\leq j\leq N}$
iff $x$ is a $(d,i)$-distributionally irregular vector for $(({\mathbb A}_{j,k})_{k\in {\mathbb N}})_{1\leq j\leq N}.$
\item[(iii)] A manifold $X'$ is a (uniformly) $(d,\tilde{X},i)$-distributionally irregular manifold for
$(({\mathcal A}_{j,k})_{k\in {\mathbb N}})_{1\leq j\leq N}$ iff $X'$ is a (uniformly) $(d,i)$-distributionally irregular manifold for the sequence
$(({\mathbb A}_{j,k})_{k\in {\mathbb N}})_{1\leq j\leq N}.$
\end{itemize}
\end{prop}

\section{From ordinary to disjoint distributional chaoticity: formulation and proof of our main structural results}\label{profsor}
First of all, we will reconsider and slightly generalize the assertions of \cite[Proposition 7, Proposition 9]{2013JFA} for disjoint sequences of single-valued linear operators in Fr\' echet and Banach spaces.

\begin{thm}\label{sequences}
Suppose that $((T_{j,k})_{k\in {\mathbb N}})_{1\leq j\leq N}$ is a tuple of operators in $L(X,Y).$ If the following two conditions are satisfied:
\begin{itemize}
\item[($I_{0,\cap}):$]
there exists a dense linear subspace
$X_{0}$ of $X$ satisfying that for each
$x\in X_{0}$ there exists a set $A_{x}\subseteq {\mathbb N}$
such that 
$\overline{dens}(A_{x})=1$ and $\lim_{k\in
A_{x}}T_{j,k}x=0,$ $1\leq j\leq N;$ 
\item[$(I_{\infty,\cap}):$] there exist a zero sequence $(y_{l})$ in $X,$ a
number $\epsilon>0,$ a strictly increasing sequence $(N_{l})$ in
${\mathbb N}$ and an integer $m\in {\mathbb N}$ such that, for every $l\in {\mathbb N},$ we have  
$$
card\Bigl( \Bigl\{1\leq k\leq N_{l} : (\forall j\in {\mathbb N}_{N})\, p_{m}^{Y}\bigl(T_{j,k}y_{l}\bigr)
>\epsilon \Bigr\}\Bigr)\geq N_{l}\bigl(1-l^{-1}\bigr),
$$
(for every $l\in {\mathbb N},$ 
$
card\bigl(\{1\leq k\leq N_{l} : (\forall j\in {\mathbb N}_{N})\, \|T_{j,k}y_{l} \|_{Y}
>\epsilon\}\bigr)\geq N_{l}(1-l^{-1}),
$
in the case that $Y$ is a Banach space), 
\end{itemize}
then there exists a $(d,1)$-distributionally irregular vector for 
$((T_{j,k})_{k\in {\mathbb N}})_{1\leq j\leq N},$ and particularly, $((T_{j,k})_{k\in {\mathbb N}})_{1\leq j\leq N}$ is $(d,1)$-distributionally chaotic.
\end{thm}

\begin{proof}
We will only outline the most relevant details of the proof. In the case that $Y$ is a
Fr\' echet space, the theorem can be simply deduced by
slightly modifying the arguments given in the proofs of \cite[Propositions 7 and 9]{2013JFA};
if $Y$ is a Banach space, then the required statement follows from the above
by endowing $Y$ with the following
increasing family of seminorms $p_{n}^{Y}(y):=n\|y\|_{Y}$ ($n\in {\mathbb N},$
$y\in Y$), which turns the space $Y$ into a linearly and topologically homeomorphic
Fr\' echet space. Concerning the above-mentioned propositions from \cite{2013JFA}, the following should be noted. First of all, for each natural number $l\in {\mathbb N}$ we set
$$
M_{l}:=\Biggl\{ x\in X: (\exists n\in {\mathbb N})\, (\forall j\in {\mathbb N}_{N})\, \frac{\bigl| \{ k\in {\mathbb N} : p_{m}^{Y}(T_{j,k}x)>l \} \cap [1,n] \bigr|}{n}\geq 1-\frac{1}{l}\Biggr\}.
$$
Then, clearly, $M_{l}$ is an open set for all $l\in {\mathbb N}$. Let $l\in {\mathbb N},$ $x\in X,$ $m_{1}\in {\mathbb N}$ and $\delta>0$. Then there exists $u\in \{y_{1},y_{2},\cdot \cdot \cdot \}$ and $n\in {\mathbb N}$ such that $p_{m_{1}}(u)<\delta \epsilon/l^{2}:=c$ and
$|\{k\in [1,n] \cap {\mathbb N} : (\forall j\in {\mathbb N}_{N}) \, p_{m}^{Y}(T_{j,k}u)>\epsilon \}|\geq n(1-\frac{1}{l}).$
Define $u_{s}:=x+ \delta s u/l^{2}c$ for $s=0,1,\cdot \cdot \cdot, l-1.$
If we replace the sets $A$ and $B_{s}$ throughout the proof of \cite[Proposition 7]{2013JFA} with the sets
$
A':=\{k\in [1,n] \cap {\mathbb N} : (\forall j\in {\mathbb N}_{N})\, p_{m}^{Y}(T_{j,k}x)>\epsilon \}
$
and $B_{s}':=\{k\in [1,n] \cap {\mathbb N} : (\exists j\in {\mathbb N}_{N})\, p_{m}^{Y}(T_{j,k}u_{s})\leq l/2\}$ ($s=0,1,\cdot \cdot \cdot, l-1$), then we can show that the set $M_{l}$ is dense. Hence, ${\mathcal M}:=\bigcap_{l\in {\mathbb N}}M_{l}$ is a residual set and each element of
${\mathcal M}$ is a $(d,1)$-distributionally $m$-unbounded vector for the sequence $((T_{j,k})_{k\in {\mathbb N}})_{1\leq j\leq N}$. Concerning \cite[Proposition 9]{2013JFA}, it is only worth noting that
\begin{align*}
 X_{0}\subseteq & M_{l,m}
\\ := & \Biggl\{ x\in X :  (\exists n\in {\mathbb N})\, (\forall j\in {\mathbb N}_{N}) \, \frac{\bigl| \{ k\in {\mathbb N} : p_{m}^{Y}(T_{j,k}x)<\frac{1}{l} \} \cap [1,n] \bigr|}{n}\geq 1-\frac{1}{l}\Biggr\}
\end{align*}
for all $l,\ m\in {\mathbb N}$ as well as that the set  $M_{l,m}$ is an open and dense subset of $X$ for all $l,\ m\in {\mathbb N},$ so that the set $\bigcap_{l,m\in {\mathbb N}}M_{l,m}$ is residual. On the other hand, this set consists exactly of $(d,1)$-distibutionally near to $0$ vectors for the sequence $((T_{j,k})_{k\in {\mathbb N}})_{1\leq j\leq N}$.  
\end{proof}

For single-valued linear operators, we use the following trick from the theory of $C$-regularized semigroups. Suppose that the condition (P) holds, where:
\begin{itemize}
\item[(P)] $T_{j} : D(T)\subseteq X \rightarrow X$ is a linear mapping, $C\in L(X)$ is an injective
mapping, as well as
$R(C)\subseteq D_{\infty}(T_{j}),$ $T_{j}^{k}C\in L(X)$ ($k\in {\mathbb N},$ $j\in {\mathbb N}_{N}$) and $T_{j,k} : R(C) \rightarrow X$ is defined by
$T_{j,k}(Cx):=T_{j}^{k}Cx,$ $x\in X,$ $k\in {\mathbb N},$ $j\in {\mathbb N}_{N}.$ 
\end{itemize}
Then,
for every $k\in {\mathbb N}$ and $j\in {\mathbb N}_{N},$ the mapping $T_{j,k} : R(C) \rightarrow X$ is an element of the space $L([R(C)],X)$. By Theorem \ref{sequences}, 
we immediately obtain that the following theorem holds good:

\begin{cor}\label{cea1}
Suppose that the condition \emph{(P)} holds, as well as that the following two conditions hold:
\begin{itemize}
\item[$(L_{0},\cap):$]
there exists a dense linear subspace
$X_{0}$ of $X$ satisfying that for each
$x\in X_{0}$ there exists a set $A_{x}\subseteq {\mathbb N}$
such that $\overline{dens}(A_{x})=1$ and $\lim_{k\in
A_{x}}T_{j}^{k}Cx=0$ ($j\in {\mathbb N}_{N}$).
\item[$(L_{\infty},\cap):$]
there exist a sequence $(z_{l})$ in $X,$ a
number $\epsilon>0,$ a strictly increasing sequence $(N_{l})$ in
${\mathbb N}$ and an integer $m\in {\mathbb N}$ such that, for every $l\in {\mathbb N},$ we have
$$
card\Bigl( \Bigl\{1\leq k\leq N_{l} : (\forall j\in {\mathbb N}_{N})\, p_{m}^{Y}\bigl(T_{j}^{k}Cz_{l}\bigr)
>\epsilon \Bigr\}\Bigr)\geq N_{l}\bigl(1-l^{-1}\bigr),
$$
(for every $l\in {\mathbb N},$
$
card\bigl(\{1\leq k\leq N_{l} : (\forall j\in {\mathbb N}_{N})\, \|T_{j}^{k}Cz_{l} \|_{Y}
>\epsilon\}\bigr)\geq N_{l}(1-l^{-1}),
$
in the case that $Y$ is a Banach space).
\end{itemize}
Then there exists a $(d,1)$-distributionally irregular vector $x\in R(C)$ for the operators $T_{1},\cdot \cdot \cdot, T_{N}.$
In particular, $T_{1},\cdot \cdot \cdot, T_{N}$ are $(d,1)$-distributionally chaotic and $\sigma$-scrambled set $S$ of $T_{1},\cdot \cdot \cdot, T_{N}$
can be chosen to be a linear submanifold of $R(C).$
\end{cor}

Using the first part of \cite[Theorem 3.7]{mendoza}, Proposition \ref{tuple-profo}, the proofs of Theorem \ref{sequences} and \cite[Propositions 7, 9]{2013JFA},
we can simply reformulate Theorem \ref{sequences} (Theorem \ref{cea1}) for any other type of $(d,i)$-distributional chaos introduced in Definition \ref{DC-unbounded-fric-DISJOINT}, by replacing optionally 
the condition $(I_{0,\cap})$ ($(L_{0,\cap})$) with one of the following conditions:
\begin{itemize}
\item[$(I_{0,\cup}):$]$=I_{0,\cap};$
\item[$(I_{0,\forall}):$] for every $j\in {\mathbb N}_{N}$ there exist a dense linear subspace
$X_{0}$ of $X$ and a set $A_{j}\subseteq {\mathbb N}$ such that
$\overline{dens}(A_{j})=1$ and
$\lim_{k\in A_{j},k\rightarrow \infty}T_{j,k}x=0;$
\item[$(I_{0,\exists}):$] there exist an integer $j\in {\mathbb N}_{N},$ a dense linear subspace
$X_{0}$ of $X$ and a set $A_{j}\subseteq {\mathbb N}$ such that
$\overline{dens}(A_{j})=1$ and
$\lim_{k\in A_{j},k\rightarrow \infty}T_{j,k}x=0;$
\item[$(L_{0},\cup):$]$=L_{0,\cap};$
\item[$(L_{0,\cup}),$ $(L_{0,\forall})$ and $(L_{0,\exists}):$] the same as $(I_{0,\cup}),$ $(I_{0,\forall})$ and $(I_{0,\exists})$, with $T_{j,k}=T_{j}^{k}C,$
\end{itemize}
and the condition $(I_{\infty,\cap})$ ($(L_{\infty,\cap})$) with one of the following conditions: 
\begin{itemize}
\item[$(I_{\infty,\cup}):$] there exist a zero sequence $(y_{l})$ in $X,$ a
number $\epsilon>0,$ a strictly increasing sequence $(N_{l})$ in
${\mathbb N}$ and an integer $m\in {\mathbb N}$ such that, for every $l\in {\mathbb N},$ we have  
$$
card\Bigl( \Bigl\{1\leq k\leq N_{l} :  \max_{1\leq j \leq N}d_{Y}\bigl(T_{j,k}y_{l},0\bigr)
>\epsilon \Bigr\}\Bigr)\geq N_{l}\bigl(1-l^{-1}\bigr),
$$
(for every $l\in {\mathbb N},$ 
$
card\bigl(\{1\leq k\leq N_{l} : \max_{1\leq j\leq N}\|T_{j,k}y_{l} \|_{Y}
>\epsilon\}\bigr)\geq N_{l}(1-l^{-1}),
$
in the case that $Y$ is a Banach space);
\item[$(I_{\infty,\forall}):$] for every $j\in {\mathbb N}_{N},$ there exist a zero sequence $(y_{l})$ in $X,$ a
number $\epsilon>0,$ a strictly increasing sequence $(N_{l})$ in
${\mathbb N}$ and an integer $m\in {\mathbb N}$ such that, for every $l\in {\mathbb N},$ we have  
\begin{align}\label{bogu-mi}
card\Bigl( \Bigl\{1\leq k\leq N_{l} :  d_{Y}\bigl(T_{j,k}y_{l},0\bigr)
>\epsilon \Bigr\}\Bigr)\geq N_{l}\bigl(1-l^{-1}\bigr),
\end{align}
(for every $l\in {\mathbb N},$ 
\begin{align}\label{bogu-miB}
card\bigl(\{1\leq k\leq N_{l} : \|T_{j,k}y_{l} \|_{Y}
>\epsilon\}\bigr)\geq N_{l}(1-l^{-1}),
\end{align}
in the case that $Y$ is a Banach space);
\item[$(I_{\infty,\exists}):$] there exist an integer $j\in {\mathbb N}_{N},$ a zero sequence $(y_{l})$ in $X,$ a
number $\epsilon>0,$ a strictly increasing sequence $(N_{l})$ in
${\mathbb N}$ and an integer $m\in {\mathbb N}$ such that, for every $l\in {\mathbb N},$ we have that \eqref{bogu-mi} holds (\eqref{bogu-miB} holds, in the case that $Y$ is a Banach space);
\item[$(L_{\infty,\cup}),$ $(L_{\infty,\forall})$ and $(L_{\infty,\exists}):$] the same as $(I_{\infty,\cup}),$ $(I_{\infty,\forall})$ and $(I_{\infty,\exists})$, with $T_{j,k}=T_{j}^{k}C.$
\end{itemize}

It is worth noting that a simple application of \cite[Proposition 8]{2013JFA} yields several equivalent conditions for the existence of a $d$-distributionally unbounded vector of type $3$ or $4$ for any sequence of linear continuous operators 
$T_{1}\in L(X),\cdot \cdot \cdot,T_{N}\in L(X)$ on a Banach space $X.$ In the present situation, we do not know how to reconsider the implication (i)' $\Rightarrow$ (i) of this statement for $d$-distributionally unbounded vectors of type $1$ or $2$ for orbits of linear continuous operators
on Banach spaces (all remaining parts of this proposition hold in our framework).

Concerning \cite[Theorem 12]{2013JFA}, it should be noted that, even in the case of considerations of orbits of linear continuous operators acting on the same Banach space $X,$ it is not possible to transfer the implication (iv) $\Rightarrow$ (i) for $(d,i)$-distributional chaos (the real problem is the validity of last equality in its proof for all $k\in {\mathbb N};$ this cannot be deduced for disjointness). But,
the proof of implication (iv) $\Rightarrow$ (iii) in \cite[Theorem 15]{2013JFA}, and the process of `renorming' described in the proof of second part of \cite[Theorem 3.7]{kerry-drew}, can be repeated almost literally in order to see that the following sufficient criterion for dense $(d,1)$-distributional chaos of linear continuous operators holds true (observe that the equations \cite[Theorem 15, (2)-(3)]{2013JFA} and the inequality preceding them hold uniformly on $j\in {\mathbb N}_{N},$ with the sequence $T^{i}$ replaced therein by
$T_{i,j}$):

\begin{thm}\label{mmnn}
Suppose that $X$ is separable, $X_{0}$ is a dense linear subspace of
$X,$ as well as $(T_{j,k})_{k\in {\mathbb N}}$ is a sequence in $L(X,Y)$ ($j\in {\mathbb N}_{N}$)
and the following holds:
\begin{itemize}
\item[(a)] $\lim_{k\rightarrow \infty}T_{j,k}x=0,$ $x\in X_{0},$ $j\in {\mathbb N}_{N},$
\item[(b)] there exists a
$(d,1)$-distributionally unbounded vector $x$ for $((T_{j,k})_{k\in {\mathbb N}})_{1\leq j\leq N}.$
\end{itemize}
Then there exists a dense uniformly $(d,1)$-distributionally irregular manifold for the sequence $((T_{j,k})_{k\in {\mathbb N}})_{1\leq j\leq N},$
and particularly, $((T_{j,k})_{k\in {\mathbb N}})_{1\leq j\leq N}$
is densely $(d,1)$-distributionally chaotic.
\end{thm}

Applying the same argumentation as in the proof of Theorem \ref{cea1} above, Theorem \ref{mmnn} implies the following:

\begin{cor}\label{mmn}
Suppose that the condition \emph{(P)} holds, $X$ is separable, $X_{0}$ is a dense linear subspace of
$X,$ as well as $T_{j}$ is a closed linear operator on $X$ ($j\in {\mathbb N}_{N}$)
and the following holds:
\begin{itemize}
\item[(a)] $\lim_{k\rightarrow \infty}T_{j}^{k}Cx=0,$ $x\in X_{0},$ $j\in {\mathbb N}_{N},$
\item[(b)] there exist $x\in X,$ $m\in {\mathbb N}$ and a set $B\subseteq {\mathbb N}$ such that
$\overline{dens}(B)=1,$ and $\lim_{k\rightarrow \infty ,k\in B}p_{m}(T_{j}^{k}Cx)=\infty,$ $j\in {\mathbb N}_{N},$ resp.
$\lim_{k\rightarrow \infty ,k\in B}\|T_{j}^{k}Cx\|=\infty,$ $j\in {\mathbb N}_{N},$ if $X$ is a Banach space.
\end{itemize}
Then there exists a uniformly $(d,1)$-distributionally irregular manifold
$W$ for the operators $T_{1},\cdot \cdot \cdot, T_{N},$ and particularly, $T_{1},\cdot \cdot \cdot, T_{N}$ are $(d,1)$-distributionally
chaotic. Furthermore, if $R(C)$ is dense in $X,$ then $W$ can be
chosen to be dense in $X$ and $T_{1},\cdot \cdot \cdot, T_{N}$ are densely $(d,1)$-distributionally
chaotic.
\end{cor}

It is worth noting that Theorem \ref{mmnn} and Corollary \ref{mmn} can be straightforwardly 
reformulated for $(d,9)$-distributional
chaos by using Proposition \ref{subspace-fric-DISJOINT}
and the second parts of \cite[Theorem 3.7, Corollary 3.12]{mendoza}. And, more to the point, for some other types od disjoint distributional chaos introduced above, the condition
(a) in the formulation of Theorem \ref{mmnn} can be replaced with the existence of an integer $j\in {\mathbb N}_{N}$ and a dense linear subspace $X_{0}$ of $X$ such that $\lim_{k\rightarrow \infty}T_{j,k}x=0,$ $x\in X_{0},$ 
and
the condition (b) in its formulation can be replaced with the condition
\begin{itemize}
\item[(b)'] there exist an integer $m\in {\mathbb N}$ and a
$d$-distributionally $m$-unbounded vector $x$ of type $2$ or $4$ for $((T_{j,k})_{k\in {\mathbb N}})_{1\leq j\leq N},$
\end{itemize}
producing the clear results.
This follows from a careful inspection of the proof of \cite[Theorem 15]{2013JFA}, by observing that it is always possible to construct an increasing sequence of natural numbers $(n_{k})_{k\in {\mathbb N}}$ and a sequence $(x_{k})_{k\in {\mathbb N}}$ in $X$ such that $p_{k}(x_{k})\leq 1$ for all $k\in {\mathbb N}$ and the inequalities \cite[(2)-(3)]{2013JFA} hold true. From the practical point of view, the consideration of dense $(d,1)$-distributional chaos in Theorem \ref{mmnn} and Corollary \ref{mmn} is of crucial importance; for example,
by using Corollary \ref{mmn} and \cite[Theorem 13]{kerry-drew}, we are in a position to present a great number of unbounded differential operators that are densely $(d,1)$-distributionally chaotic (cf. also \cite[Example 3.8, Example 3.9, Example 3.10]{kerry-drew}):

\begin{example}\label{qaz}
Suppose that $X$ is separable, ${\mathbb K}={\mathbb C},$ $A$ is a closed densely defined operator on $X,$ $z_{0}\in {\mathbb C} \setminus
\{0\},$ $\beta \geq -1,$ $d\in (0,1],$ $m\in (0,1),$ $\varepsilon
\in (0,1],$ $\gamma >-1,$ $B_{d}:=\{z\in {\mathbb C} : |z|\leq d\}$ and the following two conditions hold:
\begin{itemize}
\item[$(\S)$] $P_{z_{0},\beta, \varepsilon, m}:=e^{i\arg(z_{0})}\bigl(|z_{0}|+(P_{\beta, \varepsilon, m} \cup B_{d})\bigr) \subseteq \rho (A),$
$(\varepsilon,m(1+\varepsilon)^{-\beta}) \in \partial B_{d},$
\item[($\S \S$)] the family $\{(1+|\lambda|)^{-\gamma}(\lambda-A)^{-1} : \lambda \in P_{z_{0},\beta, \varepsilon, m}
\}$ is equicontinuous,
\end{itemize}
where $\rho(A)$ stands for the resolvent set of $A.$ 
Let $b\in
(0,1/2)$ be fixed, $\delta_{b}:=\arctan(\cos \pi b)$ and
$A_{0}:=e^{-i\arg(z_{0})}A-|z_{0}|.$ Define, for every $z\in
\Sigma_{\delta_{b}},$
\begin{equation}\label{frac}
T_{b}(z)x:=\frac{1}{2\pi i}\int
\limits_{\Gamma}e^{-z(-\lambda)^{b}}\bigl(\lambda -
A_{0}\bigr)^{-1}x\, d\lambda,\ x\in X,
\end{equation}
with the contour $\Gamma$ being defined in the proof of  \cite[Theorem 13]{kerry-drew}.
Then $T_{b}(z)\in L(X)$ is injective and has dense range in $X$ for
any $z\in \Sigma_{\delta_{b}},$ $T_{b}(z)A\subseteq
AT_{b}(z),$ $z\in \Sigma_{\delta_{b}}$ and $A^{n}T_{b}(z)\in L(X)$
($n\in {\mathbb N},$ $z\in \Sigma_{\delta_{b}}$). Furthermore, the existence of a dense subset $X_{0}$
of $X$ and a number $\lambda \in \sigma_{p}(A)$ such that
$\lim_{k\rightarrow \infty}A^{k}x=0,$ $x\in X_{0}$ and
$|\lambda|>1,$ yields that the operator
$zA^{n}$ is densely distributionally chaotic for any $n\in {\mathbb
N}$ and $z\in {\mathbb C}$ with $|z|=1.$ Suppose now that
$|z_{1}|=\cdot \cdot \cdot =|z_{N}|=1$ and $n_{1},\cdot \cdot \cdot,n_{N}\in {\mathbb N}.$
By Corollary \ref{mmn}, we get that the operators $z_{1}A^{n_{1}},\cdot \cdot \cdot, z_{N}A^{n_{N}}$ are
densely $(d,1)$-distributionally chaotic. For example, let $a,\ b,\ c>0,$ $c<\frac{b^{2}}{2a}<1,$ $X:=L^{2}([0,\infty))$ and
$$
\Lambda :=\Biggl\{ \lambda \in {\mathbb C} :
\Biggl|\lambda-\Biggl(c-\frac{b^{2}}{4a}\Biggr)\Biggr|\leq \frac{b^{2}}{4a},\
\Im \lambda \neq 0 \mbox{ if }\Re \lambda \leq c-\frac{b^{2}}{4a}
\Biggr\}.
$$
Define the operator $-B$ by $D(-B):=\{f\in
W^{2,2}([0,\infty)) : f(0)=0\}$ and $-Bu:=au_{xx}+bu_{x}+cu,$ $u\in
D(B)$ (\cite{fund}). Let $P(z)=\sum_{j=0}^{n}a_{j}z^{j}$ be a non-constant
complex polynomial such that $a_{n}>0$ and
$
P(-\Lambda) \ \cap \ S_{1} \neq
\emptyset.
$
Then the above requirements hold with $A=P(B).$
\end{example}

Without any doubt, the most commonly used condition ensuring the validity of Corollary \ref{mmn}(b) is that one in which the operator $(T_{1},\cdot \cdot \cdot, T_{N})$ has an eigenvalue $\lambda \in {\mathbb K}$ such that $|\lambda|>1,$
so that \cite[Theorem 16(II)(a), Corollary 17]{2013JFA} has a straightforward extension for $(d,1)$-distributional chaos. 
Possible applications can be given to the tuples of continuous linear operators acting on the Fr\' echet space $H({\mathbb C}^{n}),$ equipped with the usual topology of uniform 
convergence on compact sets, as well:

\begin{example}\label{lupa}
As it is well known, J. Godefroy and J. H. Shapiro has proved in \cite[Theorem 5.1]{godefroy} that any linear continuous operator $L\in L(H({\mathbb C}^{n}))$ that commutes with all translations and is not a scalar multiple of the identity is hypercyclic (now it is also known that $L$ is frequently hypercyclic, chaotic and densely distributionally chaotic, as well; see \cite{erdper} for the notion and \cite{2013JFA}). If this is the case, there exists a non-constant entire function $\Phi $ on $ {\mathbb C}^{n}$ of exponential type such that $L=\Phi(D);$ see \cite[Proposition 5.2]{godefroy} for more details. Using Theorem \ref{mmnn} and the proof of 
above mentioned theorem, it can be simply deduced that any tuple $\Phi_{1}(D),\cdot \cdot \cdot, \, \Phi_{N}(D)$ of such operators is densely $(d,1)$-distributionally chaotic, provided that $\Phi_{j}$ satisfies the properties described above ($1\leq j\leq N$) and there exists a number $\alpha \in {\mathbb C}^{n}$ such that $\Phi_{i}(\alpha)=\Phi_{j}(\alpha)$ and $|\Phi_{i}(\alpha)|>1$ for $1\leq i,\ j\leq N.$ Further analysis of growth rate of entire functions that are $(d,1)$-distributionally irregular for the differentiation operators $\Phi_{1}(D),\cdot \cdot \cdot, \, \Phi_{N}(D)$ is without scope of this paper; see L. Bernal-Gonz\' alez, A. Bonilla \cite{gonzalez} for more details about this subject.
\end{example}

Concerning the validity of conditions from Theorem \ref{mmnn}(b) and Corollary \ref{mmn}(b), we want first to note that the
following slight modification of \cite[Theorem 3]{milervrs}, discovered by V. M\"uller and J. Vr\v sovsk\'y,
holds for sequences of operators acting between possibly different Banach (Hilbert) spaces. The proof is almost the same and relies on the use of K. Ball's planck theorems
\cite[Theorems 1 and 2]{milervrs} (cf. also \cite[Proposition 9.1(a)]{meise} and \cite[Theorem 4.12]{miler-kyoto}):

\begin{prop}\label{orbits-rija}
Suppose that $X$ and $Y$ are Banach spaces, and $T_{k}\in L(X,Y)$ for each $k\in {\mathbb N}.$ If either
\begin{itemize}
\item[(i)] $\sum \limits^{\infty}_{k=1}\frac{1}{\|T_{k}\|}<\infty$
\end{itemize}
or
\begin{itemize}
\item[(ii)] $X$ is a complex Hilbert space and $\sum \limits^{\infty}_{k=1}\frac{1}{\|T_{k}\|^{2}}<\infty ,$
\end{itemize}
then there exists $x\in X$ such that $\lim_{k\rightarrow \infty}\|T_{k}x\|_{Y}=\infty .$
\end{prop}

The above proposition is clearly applicable to closed linear operators $T$ for which the following single-valued analogue of condition (P) holds:

\begin{itemize}
\item[(P$_{0}$)] $T: D(T)\subseteq X \rightarrow X$ is a linear mapping, $C\in L(X)$ is an injective
mapping, as well as
$R(C)\subseteq D_{\infty}(T),$ $T^{k}C\in L(X)$ ($k\in {\mathbb N}$) and $T_{k} : R(C) \rightarrow X$ is defined by
$T_{k}(Cx):=T^{k}Cx,$ $x\in X,$ $k\in {\mathbb N}.$ 
\end{itemize}

More precisely, we have the following

\begin{cor}\label{ora-rija}
Suppose that \emph{(P$_{0}$)} and exactly one of the following two conditions hold:
\begin{itemize}
\item[(i)] $X$ is a Banach space and $\sum \limits_{k=1}^{\infty}\frac{1}{\|T^{k}C\|}<\infty ,$
\end{itemize}
or
\begin{itemize}
\item[(ii)] $X$ is a complex Hilbert space and $\sum \limits_{k=1}^{\infty}\frac{1}{\|T^{k}C\|^{2}}<\infty .$
\end{itemize}
Then there exists $x\in X$ such that
$\lim_{k\rightarrow \infty }\|T^{k}Cx\|=\infty .$ 
\end{cor}

It is clear that Proposition \ref{orbits-rija} and Corollary \ref{ora-rija} can be directly applied for proving certain results about the existence of $d$-distributionally unbounded vectors of type $2$, $3$ or $4$ for tuples of closed linear operators. For example, the existence of a $d$-distributionally unbounded vector of type $2$ (see also Proposition \ref{tuple-profo}) for a sequence $((T_{j,k})_{k\in {\mathbb N}})_{j\in {\mathbb N}_{N}}$ in $L(X,Y)$ (the sequence $((T_{j,k})_{k\in {\mathbb N}})_{j\in {\mathbb N}_{N}}$of linear unbounded operators in $X$ satisfying the condition (P) with an injective operator $C\in L(X)$) follows by replacing the conditions (i) and (ii) in Proposition \ref{orbits-rija} (Corollary \ref{ora-rija}) with
$\sum^{\infty}_{k=1}\frac{1}{\sup_{1\leq j\leq N}\|T_{j,k}\|}<\infty$
and $\sum^{\infty}_{k=1}\frac{1}{\sup_{1\leq j\leq N}\|T_{j,k}\|^{2}}<\infty $
($\sum_{k=1}^{\infty}\frac{1}{\sup_{1\leq j\leq N}\|T_{j}^{k}C\|}<\infty $ and
$\sum_{k=1}^{\infty}\frac{1}{\sup_{1\leq j\leq N}\|T_{j}^{k}C\|^{2}}<\infty $).
Concerning $d$-distributionally unbounded vectors of type $1$, the situation is similar but not so straightforward as above:
 
\begin{prop}\label{owq-prcko}
Suppose that $X$ and $Y$ are Banach spaces, and $T_{j,k}\in L(X,Y)$ for each $k\in {\mathbb N}$ and $j\in {\mathbb N}_{N}.$ If either
\begin{itemize}
\item[(i)] for each $j\in {\mathbb N}_{N}$ we have $\sum \limits^{\infty}_{k=1}\frac{1}{\|T_{j,k}\|}<\infty$
\end{itemize}
or
\begin{itemize}
\item[(ii)] $X$ is a complex Hilbert space and for each $j\in {\mathbb N}_{N}$ we have $\sum \limits^{\infty}_{k=1}\frac{1}{\|T_{j,k}\|^{2}}<\infty ,$
\end{itemize}
then there exists $x\in X$ such that $\lim_{k\rightarrow \infty}\|T_{j,k}x\|_{Y}=\infty $ for each $j\in {\mathbb N}_{N}.$
\end{prop}

\begin{proof}
We will prove only (i).
Denote, for every natural number $k,$ by $a(k,n)$ the remainder after dividing $k$ by $N.$ Put ${\mathbf T}_{k}:=T_{a(k,N),\lceil k/N\rceil}$ for all $k\in {\mathbb N}.$ Then 
$$
\sum \limits^{\infty}_{k=1}\frac{1}{\|{\mathbf T}_{k}\|}=\sum \limits_{j=1}^{N}\sum \limits^{\infty}_{k=1}\frac{1}{\|T_{j,k}\|}<\infty,
$$
so that there exists $x\in X$ such that $\lim_{k\rightarrow \infty}\|{\mathbf T}_{k}x\|_{Y}=\infty.$ Since for each $j\in {\mathbb N}_{N}$ and $k\in {\mathbb N}$ one has $T_{j,k}x={\mathbf T}_{j+(k-1)N}x,$ we immediately get from the above that
$\lim_{k\rightarrow \infty}\|T_{j,k}x\|_{Y}=\infty ,$ as claimed.
\end{proof}

\begin{cor}\label{ora-rija-3}
Suppose that \emph{(P)} and exactly one of the following two conditions hold:
\begin{itemize}
\item[(i)] $X$ is a Banach space and for each $j\in {\mathbb N}_{N}$ we have $\sum \limits_{k=1}^{\infty}\frac{1}{\|T_{j}^{k}C\|}<\infty ,$
\end{itemize}
or
\begin{itemize}
\item[(ii)] $X$ is a complex Hilbert space and for each $j\in {\mathbb N}_{N}$ we have $\sum \limits_{k=1}^{\infty}\frac{1}{\|T_{j}^{k}C\|^{2}}<\infty .$
\end{itemize}
Then there exists $x\in X$ such that
$\lim_{k\rightarrow \infty }\|T_{j}^{k}Cx\|=\infty $ for each $j\in {\mathbb N}_{N}.$ 
\end{cor}

It is clear that Corollary \ref{mmn} in combination with Corollary \ref{ora-rija-3} (Corollary \ref{ora-rija} in combination with \cite[Corollary 3.12]{mendoza})
can be applied in the analysis of $(d,1)$-distributionally chaotic properties (distributionally chaotic properties) of linear backward shift operators in Fr\' echet sequence spaces (see Subsection \ref{shifts} below for the case in which such shift operators are continuous). 
For the sake of simplicity, we will formulate only one result connecting Corollary \ref{mmn} and Corollary \ref{ora-rija-3}, for backward shift operators in the spaces $X:=l^{p},$ where $1\leq p<\infty,$ and $X:=c_{0}.$

\begin{thm}\label{rikardinjo}
Let $X:=l^{p}$ for some $1\leq p<\infty$ or $X:=c_{0}.$
Suppose that $(w_{j,n})_{n\in {\mathbb N}}$ is a sequence of positive reals and that the unilateral weighted backward shift $T_{j}$ is given by
\begin{align}
\notag
T_{j}& \bigl\langle x_{n}\bigr\rangle_{n\in {\mathbb N}}:=\bigl\langle w_{j,n}x_{n+1}\bigr\rangle_{n\in {\mathbb N}},\quad \bigl\langle x_{n}\bigr\rangle_{n\in {\mathbb N}}\in X,\mbox{ and }
\\\label{oladilo} & D\bigl(T_{j}\bigr):=\Bigl\{ \bigl\langle x_{n}\bigr\rangle_{n\in {\mathbb N}} \in X : T_{j}\bigl\langle x_{n}\bigr\rangle_{n\in {\mathbb N}}\in X\Bigr\} \quad \bigl( j\in {\mathbb N}_{N}\bigr).
\end{align}
Suppose, further, that there exists a bounded sequence $(a_{n})_{n\in {\mathbb N}}$ of positive reals such that for each $k\in {\mathbb N}$ and $j\in {\mathbb N}_{N}$ we have
$$
B_{j,k}:=\sup_{n\in {\mathbb N}}\Biggl[a_{k+n}\prod_{i=n}^{k+n-1}\omega_{j,i}\Biggr]<\infty.
$$ 
Then the following holds:
\begin{itemize}
\item[(i)] If $p\neq 2$ and $X=l^{p}$ or $X=c_{0},$ as well as for each $j\in {\mathbb N}_{N}$ we have $\sum \limits_{k=1}^{\infty}\frac{1}{B_{j,k}}<\infty ,$
\end{itemize}
or
\begin{itemize}
\item[(ii)] $X=l^{2}$ and for each $j\in {\mathbb N}_{N}$ we have $\sum \limits_{k=1}^{\infty}\frac{1}{B_{j,k}^{2}}<\infty ,$
\end{itemize}
then the operators $T_{1},\cdot \cdot \cdot,T_{N}$ are densely $(d,1)$-distributionally chaotic.
\end{thm}

\begin{proof}
Set $C\langle x_{n}\rangle_{n\in {\mathbb N}}:=\langle a_{n} x_{n}\rangle_{n\in {\mathbb N}},$ $\langle x_{n}\rangle_{n\in {\mathbb N}} \in X.$
It is clear that $C\in L(X)$ is injective and has dense range, as well as that
$$
T_{j}^{k}C\langle x_{n}\rangle_{n\in {\mathbb N}}=\Biggl \langle a_{k+1}x_{k+1}\prod_{i=1}^{k}\omega_{j,i},a_{k+2}x_{k+2}\prod_{i=2}^{k+2}\omega_{j,i},\cdot \cdot \cdot,a_{k+n}x_{k+n}\prod_{i=n}^{k+n-1}\omega_{j,i},\cdot \cdot \cdot \Biggr \rangle,
$$
for any $\langle x_{n}\rangle_{n\in {\mathbb N}}\in X,$ $j\in {\mathbb N}_{N}$ and
$k\in {\mathbb N}.$ Hence, $T_{j}^{k}C\in L(X),$ $\|T_{j}^{k}C\|=B_{j,k}$ ($j\in {\mathbb N}_{N},$ 
$k\in {\mathbb N}$) and the condition (P) holds. 
An application of Corollary \ref{ora-rija-3} yields that there exists $x\in X$ such that
$\lim_{k\rightarrow \infty }\|T_{j}^{k}Cx\|=\infty $ for each $j\in {\mathbb N}_{N}.$ The proof of result follows now by applying Corollary \ref{mmn}, with $X_{0}$ being the linear subspace of $X$ containing sequences with only finite number of non-zero components.
\end{proof}

A concrete example of application is given as follows: 

\begin{example}\label{primena-shifts}
Let $X:=l^{p}$ for some $p\in [1,\infty)$ or $X:=c_{0}.$ For each $j\in {\mathbb N}_{N},$ we set
$\omega_{j,n}:=2^{j}n^{j}$ and $a_{n}:= (n-1)!^{-(N+1)}$ ($n\in {\mathbb N}$). Then $T_{j}^{k}C\langle x_{n}\rangle_{n\in {\mathbb N}}=\langle 2^{jk}(k+n-1)!^{j-(N+1)}(n-1)!^{j}x_{n+k}\rangle_{n\in {\mathbb N}},$ so that $T_{j}^{k}C\in L(X)$ for all $j\in {\mathbb N}_{N}$ and
$k\in {\mathbb N};$ hence, the condition (P) holds. Since $B_{j,k}\geq 2^{jk}k^{j-(N+1)}$ for all $j\in {\mathbb N}_{N}$ and
$k\in {\mathbb N},$ the operators $T_{1},\cdot \cdot \cdot,T_{N}$ are densely $(d,1)$-distributionally chaotic due to Theorem \ref{rikardinjo}.
\end{example}

There is no need to say that Theorem \ref{rikardinjo} is not universal and there exist a great number of concrete situations where it cannot be applied with any choice of regularizing operator $C$, possibly different from that one employed in Theorem \ref{rikardinjo} (see e.g. \cite[Example 4]{milervrs} and \cite[Example 5]{2013JFA}). 

\section{$(d,1)$-Distributional chaos for some special classes of operators}\label{profsor-klase}

In this section, we apply our results from previous sections to some special classes of operators on Fr\' echet and Banach spaces, like weighted backward shifts and weighted translations on locally compact groups. For the sake of brevity, we will consider only $(d,1)$-distributional chaos here.

\subsection{$(d,1)$-Distributional chaos for weighted backward shifts}\label{shifts}

Let $X$ be a Fr\' echet sequence space in which $(e_{n})_{n\in {\mathbb N}}$ is a basis (see e.g. \cite[Section 4.1]{erdper}). In this subsection, we will always assume that for each $j\in {\mathbb N}_{N}$ the unilateral weighted backward shift $T_{j}$ given by \eqref{oladilo}
is a continuous linear operator on $X.$ Since the finite linear combinations of vectors 
from the basic $(e_{n})_{n\in {\mathbb N}}$ form a dense submanifold of $X$, an application of Corollary \ref{mmn} immediately yields the following result closely connected with the assertion of \cite[Theorem 25]{2013JFA}:

\begin{prop}\label{guerrero}
Suppose that there exist $x\in X,$ $m\in {\mathbb N}$ and a set $B\subseteq {\mathbb N}$ such that
$\overline{dens}(B)=1$ and 
$\lim_{k\rightarrow \infty ,k\in B}\|T_{j}^{k}x\|_{m}=\infty,$ $j\in {\mathbb N}_{N}.$
Then there exists a dense uniformly $(d,1)$-distributionally irregular manifold
$W$ for the operators $T_{1},\cdot \cdot \cdot, T_{N},$ and particularly, $T_{1},\cdot \cdot \cdot, T_{N}$ are densely $(d,1)$-distributionally
chaotic. 
\end{prop}

An illustrative example of application is given as follows:

\begin{example}\label{primerinjo}
Let $X:=l^{p}$ for some $1\leq p<\infty$ and let $\omega_{j,n}:= \omega_{j}>1$ ($n\in {\mathbb N},$ $j\in {\mathbb N}_{N}$). Utilizing Proposition \ref{guerrero}, it can be deduced that the operators $T_{1},\cdot \cdot \cdot, T_{N}$ are densely $(d,1)$-distributionally
chaotic. Speaking-matter-of-factly, a $(d,1)$-distributionally irregular vector of type $1$ for these operators is quite simple to construct here: take 
$x:=\langle n^{-\zeta}\rangle_{n\in {\mathbb N}}$ for some finite number $\zeta>1$ 
and $B:={\mathbb N}.$ Then for each $j\in {\mathbb N}_{N}$ we have
\begin{align*}
\bigl\| T_{j}^{k}x\bigr\|^{p} \geq \sum_{n=1}^{\infty}\frac{\omega_{j}^{kp}}{(n+1+k)^{\zeta p}}\geq \sum_{n=1}^{\infty}\frac{\omega_{j}^{kp}}{(3nk)^{\zeta p}}\geq 3^{-\zeta p}\sum_{n=1}^{\infty}n^{-\zeta p}\frac{\omega_{j}^{kp}}{k^{\zeta p}}\rightarrow +\infty,\quad k\rightarrow +\infty,
\end{align*}
so that $x$ is a $(d,1)$-distributionally irregular vector of type $1$ for $T_{1},\cdot \cdot \cdot, T_{N}.$
\end{example}

Before proceeding further, we would like to note that we have not been able to tackle the problem whether the condition that $T_{1},\cdot \cdot \cdot, T_{N}$ are densely $(d,1)$-distributionally
chaotic implies the existence of a vector $x\in X,$ a number $m\in {\mathbb N}$ and a set $B\subseteq {\mathbb N}$ with the above properties, so that Proposition \ref{guerrero} only partially generalizes \cite[Theorem 25]{2013JFA}. 

\begin{example}\label{bruk}
Let $X:=c_{0}.$
Suppose that the weight $(w_{1,n})_{n\in {\mathbb N}}$ is a sequence of positive reals such that $T_{1}$ and $T_{2}$ are continuous linear operators, where $T_{2}$ is the unilateral backward shift operator associated to the weight
$(\omega_{2,n}:=w_{1,n}^{-1})_{n\in {\mathbb N}}.$ Then there are no 
increasing sequence $(k_{n})$ in ${\mathbb N}$ and a vector $x=\langle x_{n}\rangle_{n\in {\mathbb N}}\in X$ such that $\lim_{n\rightarrow \infty}\|T_{1}^{k_{n}}x\|=\lim_{n\rightarrow \infty}\|T_{2}^{k_{n}}x\|=+\infty.$
Suppose to the contrary that $\lim_{n\rightarrow \infty}\|T_{1}^{k_{n}}x\|=\lim_{n\rightarrow \infty}\sup_{s\geq 1}\omega_{1,s}\omega_{1,s+1}\cdot \cdot \cdot \omega_{1,s+k_{n}-1}|x_{k_{n}+s}|=+\infty$ and 
$\lim_{n\rightarrow \infty}\|T_{1}^{k_{n}}x\|=\lim_{n\rightarrow \infty}\sup_{s\geq 1}\omega_{2,s}\omega_{2,s+1}\cdot \cdot \cdot \omega_{2,s+k_{n}-1}|x_{k_{n}+s}|=+\infty.$ For any $n,\ s\in {\mathbb N},$ we have $\omega_{1,s}\omega_{1,s+1}\cdot \cdot \cdot \omega_{1,s+k_{n}-1}\leq 1$ or $\omega_{2,s}\omega_{2,s+1}\cdot \cdot \cdot \omega_{2,s+k_{n}-1}\leq 1,$ which clearly implies by the previous assumption that $\limsup_{n\rightarrow \infty}|x_{n}|=\infty,$ which is a contradiction. Similarly, 
$T_{1}$ and $T_{2}$ cannot be $(d,1)$-distributionally chaotic. But, any of the operators $T_{1}$ and $T_{2}$ can be distributionally chaotic. To see this, we can argue as in Example \ref{sunce}:
let the sequence of weights $\omega_{1} :=\omega:=(\omega_{k})_{k\in {\mathbb N}}$ and
the sets $[ A_{n}^{1},A_{n}^{2}]$ ($[ B_{n}^{1},B_{n}^{2}]$)
be defined 
as in the above-mentioned example. Then the vector $x:=\langle 1/n \rangle_{n\in {\mathbb N}}$
is distributionally unbounded for both operators $T_{1}$ and $T_{2}$ because
for each $n\in {\mathbb N}$ with $n\geq n_{0}$ and $n\in [ A_{n}^{1},A_{n}^{2}]$ ($n\in [ B_{n}^{1},B_{n}^{2}]$) we have $\omega_{2,1}\omega_{2,2}\cdot \cdot \cdot \omega_{2,n}>2^{n}$ ($\omega_{1,1}\omega_{1,2}\cdot \cdot \cdot \omega_{1,n}>2^{n}$), so that
$\|T_{1}^{n}x\|\geq \omega_{1,1}\omega_{1,2}\cdot \cdot \cdot \omega_{1,n}x_{n+1}>2^{n}/(n+1), $ $n\in [ B_{n}^{1},B_{n}^{2}]$ and $\|T_{2}^{n}x\|\geq \omega_{2,1}\omega_{2,2}\cdot \cdot \cdot \omega_{2,n}x_{n+1}>2^{n}/(n+1), $ $n\in [ A_{n}^{1},A_{n}^{2}].$
\end{example}

Now we will state and prove the following proper generalization of \cite[Theorem 26]{2013JFA} for disjoint backward shift operators:

\begin{prop}\label{da-se-ohladi}
Suppose that the operator $T$ is given by \eqref{oladilo} with the weight $w_{j,n}\equiv 1$ ($j\in {\mathbb N}_{N},$ $n\in {\mathbb N}$). Let there exist an infinite set  $S$ of natural numbers such that
the series $\sum_{n\in S}e_{n}$ converges in $X,$ and let there exist natural numbers $r_{1},\cdot \cdot \cdot,r_{N}$ such that the upper density of set
$$
{\mathrm Q}:=\Bigl\{ k\in {\mathbb N} : \bigl(\forall j\in {\mathbb N}_{N}\bigr)\, r_{j}k\in S+1 \Bigr\}
$$
is equal to one. Then the operators $T^{r_{1}},\cdot \cdot \cdot, T^{r_{N}}$ are densely $(d,1)$-distributionally
chaotic. 
\end{prop}

\begin{proof}
By Theorem \ref{mmnn}, it suffices to show that there exist $x\in X,$ $m\in {\mathbb N}$ and a set $B\subseteq {\mathbb N}$ such that
$\overline{dens}(B)=1$ and 
$\lim_{k\rightarrow \infty ,k\in B}\|T^{r_{j}k}x\|=\infty,$ $j\in {\mathbb N}_{N}.$ For this, we will essentially apply Theorem \ref{sequences} in the following way: Set $y_{l}:=\sum_{n\in S,n\geq l}e_{n},$ $l\in {\mathbb N}.$ Then $\lim_{l\rightarrow \infty}y_{l}=0$ and there exists a number $\epsilon>0$ such that $d(\langle x_{n}\rangle_{n\in {\mathbb N}},0)<\epsilon$
implies $|x_{1}|<1.$ To verify that
$(I_{\infty,\cap})$ holds, it suffices to construct 
a strictly increasing sequence $(N_{l})$ in
${\mathbb N}$ such that, for every $l\in {\mathbb N},$ we have  
$$
card\Bigl( \Bigl\{1\leq k\leq N_{l} : (\forall j\in {\mathbb N}_{N})\, d\bigl(T^{r_{j}k}y_{l},0\bigr)
>\epsilon \Bigr\}\Bigr)\geq N_{l}\bigl(1-l^{-1}\bigr).
$$
But, this simply follows from our choice of number $\epsilon ,$ the equality $\overline{dens}({\mathrm Q})=1$ and the fact that for a number $l\in {\mathbb N}$ given in advance we have
$(T^{r_{j}})^{k}y_{l}=T^{r_{j}k}y_{l}=e_{1}+\cdot \cdot \cdot$ for every $j\in {\mathbb N}_{N}$ and $k\in {\mathbb N}$ 
such that $r_{j}k-1\in S \cap [l,\infty).$ 
\end{proof}

It is straightforward to apply Proposition \ref{da-se-ohladi} with $S={\mathbb N},$ for backward shift operators in K\"othe sequence spaces (\cite{gimenez-p}). For example, if $\langle a_{n}\rangle_{n\in {\mathbb N}}$
is a sequence of positive reals,
$X:=l^{p}(\langle a_{n}\rangle_{n\in {\mathbb N}}):=\{\langle x_{n}\rangle_{n\in {\mathbb N}}\in {\mathbb C}^{{\mathbb N}} : \|\langle x_{n}\rangle_{n\in {\mathbb N}}\|_{l^{p}(\langle a_{n}\rangle)}:=(\sum_{n=1}^{\infty}|x_{n}a_{n}|^{p})^{1/p}<\infty \}$ and $\sum_{n=1}^{\infty}a_{n}^{p}<\infty,$ then for every choice of natural numbers $r_{1},\cdot \cdot \cdot,r_{N}$
the operators $T^{r_{1}},\cdot \cdot \cdot, T^{r_{N}}$ will be densely $(d,1)$-distributionally
chaotic. 

In what follows, we will generalize \cite[Corollary 27]{2013JFA} for the operators of form
$$
B_{j}\bigl\langle x_{n}\bigr\rangle_{n\in {\mathbb N}}:=\bigl\langle w_{n}^{j}x_{n+a(n,j)}\bigr\rangle_{n\in {\mathbb N}},\quad \bigl\langle x_{n}\bigr\rangle_{n\in {\mathbb N}}\in X\quad \bigl( j\in {\mathbb N}_{N}\bigr), 
$$
where, for each fixed number $j\in {\mathbb N}_{N},$ $(\omega_{n}^{j})_{n\in {\mathbb N}}$ is a sequence of scalars from the field ${\mathbb K}$ and $(a(n,j))_{n\in {\mathbb N}}$ is an increasing sequence of natural numbers (a generalization is proper even for single operators).
We will assume that $B_{j}\in L(X)$ 
for all $j\in {\mathbb N}_{N}.$ 

To formulate our result, we need some preliminaries about the computing individual orbits of the operators $B_{j}.$ Let the numbers $j\in {\mathbb N}_{N},$ $k,\ l,\ n \in {\mathbb N}$ be given, and let $n\geq l.$ Then, if exists a number $i\in {\mathbb N}$ such that $i+a(i,j)=n,$ then it is uniquely determined. Let
$b_{n}\in {\mathbb K}\setminus \{0\}$ for all $n\in {\mathbb N}.$ Then we have
\begin{align*}
B_{j}^{k}b_{n}e_{n}=b_{n}B_{j}^{k}e_{n}=b_{n}B_{j}^{k-1}\left\{
\begin{array}{l}
0,\quad \mbox{ if there is no } i \in {\mathbb N}\mbox{ such that }i+a(i,j)=n,\\
\omega_{i}^{j}e_{i},\mbox{ if }i \in {\mathbb N}\mbox{ and }i+a(i,j)=n.
\end{array}
\right.
\end{align*}
Put $c_{1}(n,j):=i,$ if $i\in {\mathbb N}$ and $i+a(i,j)=n;$ $c_{1}(n,j):=+\infty,$
otherwise. In the first case, we have $B_{j}^{k}b_{n}e_{n}=b_{n}\omega_{c_{1}(n,j)}^{j}B_{j}^{k-1}e_{c_{1}(n,j)}.$ Set $c_{2}(n,j):=i,$ if $i\in {\mathbb N}$ and $i+a(i,j)=c_{1}(n,j);$ $c_{2}(n,j):=+\infty,$
otherwise. Again, in the first case, we have $B_{j}^{k}b_{n}e_{n}=b_{n}\omega_{c_{1}(n,j)}^{j}\omega_{c_{2}(n,j)}^{j}B_{j}^{k-2}e_{c_{2}(n,j)}.$ Assume that there exists a finite sequence $c_{1}(n,j),\cdot \cdot \cdot, c_{k}(n,j)$ of natural numbers such that
$c_{k}(n,j)=1.$ In this case, we have $B_{j}^{k}b_{n}e_{n}=b_{n}\prod^{k}_{s=1}\omega_{c_{s}(n,j)}^{j}e_{1}.$ To finish this technical part, for given numbers $j\in {\mathbb N}_{N}$ and $k\in {\mathbb N},$ denote by $P_{j,k}$ the set consisting of all integres $n\in {\mathbb N}$ such that
there exists a finite sequence $c_{1}(n,j),\cdot \cdot \cdot, c_{k}(n,j)$ of natural numbers with the properties that
$c_{k}(n,j)=1$ and $b_{n}\prod^{k}_{s=1}\omega_{c_{s}(n,j)}^{j}e_{1}=1.$

Arguing as in the proof of Proposition \ref{da-se-ohladi}, with the vectors $y_{l}:=\sum_{n\in S,n\geq l}b_{n}e_{n}$ ($l\in {\mathbb N}$) and the same choice of number $\epsilon>0$, we can now formulate the following proposition:

\begin{prop}\label{jebi-ga-hak}
Let for each fixed number $j\in {\mathbb N}_{N},$ $(\omega_{n}^{j})_{n\in {\mathbb N}}$ be a given sequence of scalars from the field ${\mathbb K}$ and let $(a(n,j))_{n\in {\mathbb N}}$ be a given increasing sequence of natural numbers.
Suppose that there exist an infinite set $S$ of natural numbers such that
the series $\sum_{n\in S}b_{n}e_{n}$ converges in $X,$ and that the upper density of set
$$
{\mathrm Q}_{g}:=\Bigl\{ k\in {\mathbb N} : \bigl(\forall j\in {\mathbb N}_{N}\bigr)P_{j,k}\cap S\neq \emptyset \Bigr\}
$$
is equal to one. Then the operators $B_{1},\cdot \cdot \cdot, B_{N}$ are densely $(d,1)$-distributionally
chaotic. 
\end{prop}

It is beyond the scope of this paper to extend the assertions of \cite[Theorem 4.8(c)]{erdper} and \cite[Corollary 28]{2013JFA}, provided that $(e_{n})_{n\in {\mathbb N}}$ is an unconditional basis of $X$  
(concerning these and previously stated results, we want only to stress that similar statements hold for bilateral backward shifts (see \cite[Theorem 29, Theorem 30, Corollary 31]{2013JFA}), while the same comment and problem can be posed for \cite[Corollary 32]{2013JFA}). 
The interested reader might be also interested in the analysis of disjoint distributionally chaotic properties of composition operators (see \cite{2013JFA}, the doctoral dissertation of \"O. Martin \cite{ma10}, the paper \cite{jusef} by Z. Kamali, B. Yousefi and references cited therein for more details about this subject) and reconsideration of structural results from the paper \cite{gimenez-p} by F. Mart\'inez-Gim\'enez, P. Oprocha and A. Peris for disjoint distributional chaos.

\subsection{$(d,1)$-Distributional chaos for weighted translations on locally compact groups}\label{kuo-jo}

Here, we continue our recent research analyses raised in \cite{chen-marek}-\cite{chen-chen}. In the first part, ending with Theorem \ref{hip-on} and its proof, we assume that $G$ is a second countable locally compact group with a right invariant Haar measure
$\lambda$, and by $L^p(G)\ (1\leq p <\infty)$ we denote the Lebesgue space with respect to
$\lambda ,$ over the field of scalars ${\mathbb K}.$ By $\|\cdot\|_{p}$ we denote the norm on $L^p(G).$
A bounded continuous function $w:G \rightarrow (0,\infty)$ is
called a weight on $G$. Let $a_{j} \in G$ and let $\delta_{a_{j}}$ be
the unit point mass at $a_{j}$ ($ j\in {\mathbb N}_{N}$). A weighted translation operator on $G$ is
a weighted convolution operator $T_{j}\equiv T_{a_{j},w_{j}} : L^p(G) \rightarrow
L^p(G)$ defined by
$$T_{a_{j},w_{j}}(f) := w_{j}T_{a_{j}}(f), \quad f \in L^p(G) \quad  \bigl(j\in {\mathbb N}_{N}\bigr),$$
where $w_{j}$ is a weight on $G$ and $T_{a_{j}}(f)=f*\delta_{a_{j}}\in L^p(G)$ is
the convolution:
$$\bigl(f*\delta_{a_{j}}\bigr)(x) := \int_G f\bigl(xy^{-1}\bigr)d\delta_{a_{j}}(y) = f\bigl(xa_{j}^{-1}\bigr),
\quad x\in G \quad \bigl(j\in {\mathbb N}_{N}\bigr).$$
To simplify notations, we set
$$\varphi_{j,n}:=\prod_{s=1}^{n}w_{j}\ast\delta_{a_{j}^{-1}}^{s},\quad  n\in {\mathbb N},\ j\in {\mathbb N}_{N}.$$

We have the following theorem:

\begin{thm}\label{hip-on}
Suppose that there exist a sequence of compact sets $(K_n)_{n\in {\mathbb N}}$ in $L^p(G)$ with positive measures, a subset $B\subseteq \mathbb{N}$ with $\overline{{\rm dens}}(B)=1,$ and a sequence $(c_n)_{n\in B}$ of scalars from ${\mathbb K}$ such that
$\lim_{n\rightarrow \infty}\|\varphi_{j,n}{_{|K_{k}}}\|_{p}=0$ for all $k\in {\mathbb N}$ and $ j\in {\mathbb N}_{N}$,
as well as
\begin{align}\label{prckojed}
\sum_{n\in B}\bigl|c_n\bigr|\lambda \bigl(K_n\bigr)^{\frac{1}{p}}<\infty
\end{align}
and
\begin{align}\label{prckojed-prim}
\lim_{n\in B}\Biggl\|  \sum_{k\in B}c_{k}T_{j}^{n}\chi_{K_k} \Biggr\|_{p}=\infty,\ j\in {\mathbb N}_{N}.
\end{align}
Set $X_{0}:=span\{\chi_{K_{n}} : n\in {\mathbb N}\}$
and
$\tilde{X}:=\overline{X_{0}},$ where $\chi_{\cdot}$ denotes the characteristic function.
Then the operators $T_{1},\cdot \cdot \cdot,\ T_{N}$ are densely $(d,\tilde{X},1)$-distributionally chaotic.
\end{thm}

\begin{proof}
By the proof of \cite[Theorem 2.5]{chen-chen} and \eqref{prckojed}, the series $y:=\sum_{n\in B}c_{n}\chi_{K_n}$ absolutely converges in $\tilde{X}.$ On the other hand, it is clear that
\begin{align*}
\bigl\|T_{j}^{n}& y\bigr \|_{p}=\Biggl\| T_{j}^{n} \sum_{k\in B}c_{k}\chi_{K_k} \Biggr\|_{p}=\Biggl\|  \sum_{k\in B}c_{k}T_{j}^{n}\chi_{K_k} \Biggr\|_{p},
\end{align*}
so that an application of \eqref{prckojed-prim} gives that $\lim_{n\in B}\|T_{j}^{n}y\|_{p}=\infty$ for each $j\in {\mathbb N}_{N}.$ Furthermore, since $\lim_{n\rightarrow \infty}\|\varphi_{j,n}{_{|K_{k}}}\|_{p}=0$ for all $k\in {\mathbb N}$ and $ j\in {\mathbb N}_{N}$, the proof of afore-mentioned theorem shows that,  for every $f\in X_{0}$ and $j\in {\mathbb N}_{N},$ we have $\lim_{n\rightarrow \infty}T_{j}^{n}f=0.$ Therefore, the statement follows by applying Theorem \ref{mmnn}, which yields that the operators $(T_{1})_{|\tilde{X}} : \tilde{X}\rightarrow X,\cdot \cdot \cdot,\ (T_{N})_{|\tilde{X}} : \tilde{X}\rightarrow X$  are densely $(d,1)$-distributionally chaotic, and Proposition \ref{subspace-fric-DISJOINT}.
\end{proof}

In the remaining part of paper, we work with the weighted translation operators on the Orlicz space of a locally compact group $G$ with the identity $e$ and a right Haar measure $\lambda$. 
To achieve our aims, we recall the following notion:
a continuous, even and convex function $\Phi:{\Bbb R}\rightarrow {\Bbb R}$ is said to be a Young function iff $\Phi$ satisfies $\Phi(0)=0$, $\Phi(t)>0$ for $t>0$, and $\lim_{t\rightarrow\infty}\Phi(t)=\infty$.
The complementary function $\Psi$ of a Young function $\Phi$ is defined by
$$\Psi(y):=\sup\{x|y|-\Phi(x):x\geq 0\}$$
for $y\in\Bbb{R}$, which is also a Young function. If $\Psi$ is the complementary function of $\Phi$, then
$\Phi$ is the complementary function of $\Psi$, and the Young inequality
$$xy\leq \Phi(x)+\Psi(y)$$
holds for $x,\ y\geq0$. Then, for any Borel function $\Phi$, the set
$$L^\Phi(G):=\left\{f:G\rightarrow \Bbb{K}: \int_G\Phi(\alpha|f|)d\lambda<\infty\ \text{for some $\alpha>0$} \right\}$$
is called the Orlicz space, which is a Banach space under the Luxemburg norm $N_\Phi$,
defined for $f\in L^\Phi(G)$ by
$$N_\Phi(f):=\inf\left\{k>0:\int_G\Phi\left(\frac{|f|}{k}\right)d\lambda\leq1\right\}.$$
The Orlicz space is a generalization of the usual Lebesgue space $L^p(G)$ considered above;
more precisely, if $\Phi(t):=\frac{|t|^p}{p}$, then the Orlicz space $L^\Phi(G)$ is the Lebesgue space $L^p(G)$ ($1\leq p<\infty$).

Over the past decades, the important properties and interesting structures of Orlicz spaces have been investigated intensely
by many authors; cf. \cite{chen-marek} and references cited therein for more details on the subject.
We assume that $G$ is second countable and $\Phi$ is $\Delta_2$-regular, which means that
there exist two finite constants $M>0$ and $t_0>0$ such that $\Phi(2t)\leq M\Phi(t)$ for $t\geq t_0$ when $G$ is compact, and $\Phi(2t)\leq M\Phi(t)$ for all $t>0$ when $G$ is noncompact. For instance, the Young functions $\Phi$ given by
$$\Phi(t):=\frac{|t|^p}{p}\quad(1\leq p<\infty)\qquad \mbox{and}\qquad\Phi(t):=|t|^\alpha(1+|\log|t||)\quad (\alpha>1)$$
are both $\Delta_2$-regular. In the case that $\Phi$ is $\Delta_2$-regular, then the space $C_c(G)$ of all continuous functions on $G$ with compact support is dense in $L^\Phi(G)$.

We define the operators $T_{j}$ for $1\leq j\leq N$ as in the first part of this subsection. Then we have the following assertion, whose proof is very similar to that one of Theorem \ref{hip-on} and which relies upon the fact that $C_c(G)$ is dense in $L^\Phi(G):$ 

\begin{thm}\label{qwea}
Suppose that there exist an absolutely summable sequence $(c_n)_{n\in B}$ of scalars from ${\mathbb K}$ and a subset $B\subseteq \mathbb{N}$ with $\overline{{\rm dens}}(B)=1$ such that for each compact set $K$ of $G$ the following holds:
\begin{align}\label{prckojed-primonja}
\lim_{n\in B}N_{\Phi}\Biggl( \sum_{k\in B}c_{k}T_{j}^{n}\chi_{K}\Biggr) =\infty,\ j\in {\mathbb N}_{N}.
\end{align}
Then the operators $T_{1},\cdot \cdot \cdot,\ T_{N}$ are densely $(d,1)$-distributionally chaotic.
\end{thm}

Although we have already exhibited a great deal of comments and observations about problems considered, many other issues and questions can be proposed for disjoint distributional chaos. We want only to raise the problem of finding some sufficient conditions ensuring the validity of
\eqref{prckojed-prim} and \eqref{prckojed-primonja}.  

Finally, we would like to note that
the notion of 
$(\lambda, \tilde{X})$-distributional chaos (reiteratively
$\tilde{X}$-distributional chaos) of type $s$, where $\lambda \in (0,1]$ and $s\in \{1, 2, 2\frac{1}{2}, 3\}$, has been recently introduced in \cite{marek-trio}. Disjoint $(\lambda, \tilde{X})$-distributional chaos (disjoint reiteratively
$\tilde{X}$-distributional chaos) of type $s$ will be considered somewhere else.

\end{document}